%% file: doc.tex
\documentclass{article} %

\usepackage[
  margin=3cm,
  includefoot,
  footskip=30pt,
]{geometry}

\usepackage[utf8]{inputenc} %
\usepackage[T1]{fontenc}    %
\usepackage[pagebackref]{hyperref}
\usepackage{url}            %
\usepackage{booktabs}       %
\usepackage{amsfonts}       %
\usepackage{nicefrac}       %
\usepackage{microtype}      %
\usepackage{xcolor}         %

\definecolor{darkgreen}{rgb}{0.0,0.5,0.0}

\hypersetup{
	colorlinks=true,        %
	linkcolor=blue,         %
	filecolor=magenta,
	citecolor=darkgreen,         %
	urlcolor=blue           %
}
\renewcommand*{\backrefalt}[4]{%
    \ifcase #1 \footnotesize{(Not cited.)}%
    \or        \footnotesize{(Cited on page~#2)}%
    \else      \footnotesize{(Cited on pages~#2)}%
    \fi}

\input{macros.tex}

\usepackage{mathtools}

\usepackage{natbib}

\title{Non-convex Stochastic Composite Optimization with Polyak Momentum}

\author{
  Yuan Gao\thanks{CISPA Helmholtz Center for Information Security, Germany.} \thanks{Universit\'at des Saarlandes, Germany.} \\
    \texttt{yuan.gao@cispa.de}
    \and
    Anton Rodomanov\footnotemark[1]\\
    \texttt{anton.rodomanov@cispa.de}
    \and
    Sebastian U. Stich\footnotemark[1] \\
    \texttt{stich@cispa.de}
}

\usepackage[font=small]{caption}
\usepackage{cleveref}
\begin{document}

\maketitle

\begin{abstract}
  The stochastic proximal gradient method is a powerful generalization of the widely used stochastic gradient descent (\algname{SGD}) method and has found numerous applications in Machine Learning. 
  However, it is notoriously known that this method fails to converge in non-convex settings where the stochastic noise is significant (i.e.\ when only small or bounded batch sizes are used).
  In this paper, we focus on the stochastic proximal gradient method with Polyak momentum.
  We prove this method attains an optimal convergence rate for non-convex composite optimization problems, regardless of batch size.
  Additionally, we rigorously analyze the variance reduction effect of the Polyak momentum in the composite optimization setting
  and we show the method also converges when the proximal step can only be solved inexactly. Finally, we provide numerical experiments to validate our theoretical results.
\end{abstract}

\input{main.tex}

\end{document}

%% file: macros.tex
\usepackage{amsfonts}
\usepackage{amsmath}
\usepackage{amsthm}
\usepackage{amssymb}
\usepackage{dsfont}
\usepackage{comment}
\usepackage{multirow}
\usepackage{multicol}
\usepackage{threeparttable}

\usepackage{xcolor}
\usepackage{color}
\usepackage{graphicx}
\usepackage{pifont}%
\newcommand{\algname}[1]{{\sf#1}\xspace}
\newcommand{\eqdef}{\coloneqq}

\usepackage{verbatim}
\usepackage{xspace} %

\usepackage{enumerate}
\usepackage{enumitem}

\usepackage{subcaption}

\providecommand{\algname}[1]{{\sf#1}\xspace}

\providecommand{\norm}[1]{\left\lVert#1\right\rVert}

  \providecommand{\R}{\mathbb{R}} %

  \providecommand{\wtilde}[1]{\widetilde{#1}}
  \providecommand{\Eb}[1]{{\mathbb E}\left[#1\right] }       %

  \DeclareMathOperator*{\argmin}{arg\,min}

  \providecommand{\0}{\mathbf{0}}

  \renewcommand{\gg}{\mathbf{g}}

  \providecommand{\mm}{\mathbf{m}}

  \providecommand{\xx}{\mathbf{x}}
  \providecommand{\yy}{\mathbf{y}}

  \providecommand{\mI}{\mathbf{I}}

  \providecommand{\cN}{\mathcal{N}}
  \providecommand{\cO}{\mathcal{O}}

  \usepackage{bm}

\RequirePackage[colorinlistoftodos,bordercolor=orange,backgroundcolor=orange!20,linecolor=orange,textsize=scriptsize]{todonotes}
\providecommand{\mycomment}[3]{\todo[caption={},color=#3!20,inline]{\textbf{#1: }#2}}%
\providecommand{\myinlinecomment}[3]{%
  {\color{#1}#2: #3}}%
\newcommand\commenter[2]%
{%
  \expandafter\newcommand\csname i#1\endcsname[1]{\myinlinecomment{#2}{#1}{##1}}
  \expandafter\newcommand\csname #1\endcsname[1]{\mycomment{#1}{##1}{#2}}
}

\newtheorem{proposition}{Proposition}
\newtheorem{lemma}{Lemma}

\newtheorem{remark}[lemma]{Remark}
\newtheorem{assumption}{Assumption}
\newtheorem{theorem}{Theorem}

\usepackage{url}

\renewcommand{\epsilon}{\varepsilon}

\usepackage{algorithm}
\usepackage[noend]{algpseudocode}

\errorcontextlines\maxdimen

\makeatletter
    \newcommand*{\algrule}[1][\algorithmicindent]{\makebox[#1][l]{\hspace*{.5em}\thealgruleextra\vrule height \thealgruleheight depth \thealgruledepth}}%
\newcommand*{\thealgruleextra}{}
\newcommand*{\thealgruleheight}{.75\baselineskip}
\newcommand*{\thealgruledepth}{.25\baselineskip}

\newcount\ALG@printindent@tempcnta
\def\ALG@printindent{%
    \ifnum \theALG@nested>0%
        \ifx\ALG@text\ALG@x@notext%
        \else
            \unskip
            \addvspace{-1pt}%
            \ALG@printindent@tempcnta=1
            \loop
                \algrule[\csname ALG@ind@\the\ALG@printindent@tempcnta\endcsname]%
                \advance \ALG@printindent@tempcnta 1
            \ifnum \ALG@printindent@tempcnta<\numexpr\theALG@nested+1\relax%
            \repeat
        \fi
    \fi
    }%
\usepackage{etoolbox}
\patchcmd{\ALG@doentity}{\noindent\hskip\ALG@tlm}{\ALG@printindent}{}{\errmessage{failed to patch}}
\makeatother

\newbox\statebox
\newcommand{\myState}[1]{%
    \setbox\statebox=\vbox{#1}%
    \edef\thealgruleheight{\dimexpr \the\ht\statebox+1pt\relax}%
    \edef\thealgruledepth{\dimexpr \the\dp\statebox+1pt\relax}%
    \ifdim\thealgruleheight<.75\baselineskip
        \def\thealgruleheight{\dimexpr .75\baselineskip+1pt\relax}%
    \fi
    \ifdim\thealgruledepth<.25\baselineskip
        \def\thealgruledepth{\dimexpr .25\baselineskip+1pt\relax}%
    \fi
    \State #1%
    \def\thealgruleheight{\dimexpr .75\baselineskip+1pt\relax}%
    \def\thealgruledepth{\dimexpr .25\baselineskip+1pt\relax}%
}

\usepackage{mathtools}
\DeclarePairedDelimiterX{\inp}[2]{\langle}{\rangle}{#1, #2}
\DeclarePairedDelimiterX{\cbr}[1]{\{}{\}}{#1} %
\DeclarePairedDelimiterX{\rbr}[1]{(}{)}{#1} %
\DeclarePairedDelimiterX{\sbr}[1]{[}{]}{#1} %

\DeclarePairedDelimiter{\ceil}{\lceil}{\rceil}

\usepackage{thm-restate}

\usepackage{cleveref}
\crefname{assumption}{Assumption}{Assumptions}
\Crefname{assumption}{Assumption}{Assumptions}

%% file: main.tex
\section{Introduction}

Stochastic gradient descent and its variants are the workhorse of modern machine learning. The stochastic proximal gradient method is a simple yet powerful extension of the vanilla stochastic gradient descent method, which aims to solve the following stochastic composite optimization problem:
\begin{equation}
    \label{eq:composite}
    \min_{\xx\in\R^d}\{F(\xx)\eqdef  f(\xx)+\psi(\xx)\} \,,
\end{equation}
where $ f:\R^d\to \R$ is smooth and $\psi \colon \R^d \to \R \cup \{+ \infty\}$ is not necessarily differentiable, but a simple function.
Such paradigm is ubiquitous in machine learning and beyond, and it covers a wide range of generalizations of the vanilla optimization problem, including regularized machine learning problems~\cite{Liu_2015_CVPR}, signal processing~\cite{combettes2010proximal}, image processing~\cite{Luke2020} and many more. It also naturally covers constrained optimization problems by considering the indicator function of the constraint set, and some recent variants of such ideas have been applied to distributed and federated machine learning problems~\cite {mishchenko2022proxskip}. 

The problem formulation~\eqref{eq:composite} is often used in Machine Learning as the proximal term $\psi$ allows encoding prior domain knowledge. For instance, it can impose constraints on the variables, handle non-differentiability, induce sparsity, or it can take the form of a regularizer. 
Recently, there has been a surge of interest in these \emph{preconditioning} techniques in the Deep Learning context~\cite{hendrikx2020statistically,woodworth2023losses}. 
Given a loss objective function $\ell \colon \R^d \to \R$, an auxiliary function (or preconditioner) $\hat \ell$ is constructed---for instance, by approximating $\ell$ with a subset of the data samples or synthetic samples. This results in the formulation:
\begin{equation*}
    \min_{\xx\in\R^d}\{ \underbrace{\ell(\xx)-\hat \ell(\xx)}_{f(\xx)} +  \underbrace{\hat \ell(\xx)}_{\psi(\xx)}\} \,.
\end{equation*}
This is a special case of Problem~\eqref{eq:composite}, with $f\eqdef \ell-\hat \ell$ and $\psi\eqdef \hat \ell$. 
This formulation has the advantage that the training can be accelerated if the proxy function $\hat \ell$ remains simple to optimize.
These applications motivate the need to further our understanding of the stochastic composite optimization problem in the non-convex regime, especially where large batches are less preferred or even unavailable~\cite{lecun2012efficient,rieke2020future}.

In the convex and strongly convex case, the complexity of solving Problem~\eqref{eq:composite} is well understood~\cite{ghadimi2016mini,ghadimi2013accelerated}. On the other hand, the existing theory in the non-convex regime is unsatisfactory. In particular, when the gradient noise has variance $\sigma^2$, with the vanilla algorithm, the squared norm of the gradient of $F$ is only shown to converge to $\cO(\nicefrac{1}{K}) + \Omega(\sigma^2)$ after $K$ iterations. This implies that, to converge to an arbitrary $\epsilon$ error, one needs to take mega batches of size $\Omega(\epsilon^{-1})$ at each iteration of the algorithm to reduce the stochastic noise term; and the total number of stochastic gradient oracle calls is $\Omega(\epsilon^{-2})$. In practice and theory, smaller-batch methods are often preferred over mega-batch methods. For example, mega-batch methods follow the full gradient methods much more closely at each step than smaller batch methods, and it is observed empirically to be adversarial to the generalization performance~\cite{wilson003general,lecun2012efficient,  keskar2017largebatch}, and theoretically~\cite{sekhari2021sgd} for certain Machine Learning tasks. Furthermore, in many practical settings, mega-batches are unavailable or intractable to sample, e.g., in medical tasks~\cite{rieke2020future}; federated Reinforcement Learning~\cite{khodadadian2022federated, jin2022federated}; and multi-agent Reinforcement Learning~\cite{doan2019finite}. 

In this work, we revisit the stochastic composite optimization problem in the non-convex regime and show that the Polyak momentum technique addresses the aforementioned mega-batch issue while retaining the optimal convergence rate $\cO(\epsilon^{-2})$. Polyak momentum is the de-facto standard for training Deep Learning models in practice~\cite{kingma2014adam}, and understanding its theoretical efficacy in various settings is an active research direction~\cite{cutkosky2020momentum, fatkhullin2023momentum}.

\subsection{Stochastic Proximal Gradient Method}
\label{sec:spg}
To solve the stochastic composite optimization problem~\eqref{eq:composite}, we consider the stochastic proximal gradient method, which solves the following subproblem in each iteration (commonly known as the proximal step):
find $\xx_{k+1}$ such that
\[
    \Omega_k(\xx_{k+1})\leq \Omega_k(\xx_k) \quad \text{and} \quad \nabla \Omega_k(\xx_{k+1}) = 0,
\]
where
$
    \Omega_k(\xx) \eqdef \inp{\gg_k}{\xx}+\psi(\xx)+\frac{ M_k }{2}\norm{\xx-\xx_k}^2
$.
Here, $\gg_k$ denotes a random vector computed at each iteration $k$ using the stochastic first-order oracle of $f$, and $M_k>0$ is a regularization (or stepsize) parameter. $M_k$ can be a user defined constant, or it can be a non-decreasing sequence.
For the vanilla stochastic proximal gradient method, $\gg_k$ is a stochastic gradient of $f$ at $\xx_k$. 

When $\psi$ is convex, the proximal step is equivalent to minimizing~$\Omega_k$:
\begin{equation}
    \label{eq:proximal-step}
    \xx_{k+1} = \argmin_{\xx \in \R^d} \Bigl\{ \inp{\gg_k}{\xx}+\psi(\xx)+\frac{ M_k }{2}\norm{\xx-\xx_k}^2 \Bigr\}.
\end{equation}
We note that when $\psi\equiv 0$, this reduces to the vanilla \algname{SGD} with step size $\frac{1}{M_k}$. When $\psi\not\equiv 0$, it has long been known that if error-dependent batch sizes are not allowed, then such a method is only proved to converge to a neighborhood of the stationary point of $F$ up to the variance of the gradient noise~\cite{ghadimi2016mini}. In \Cref{sec:lower-bound} we give a simple example to illustrate this phenomenon.
Therefore, one needs $\epsilon$-dependent mini-batches (or mega-batches) to converge to an $\epsilon$ neighborhood of the stationary point of $F$. 

\subsection{Our Contributions}

Stochastic Polyak momentum has seen a huge success in practice~\cite{kingma2014adam}. In this work, we study the effect of Polyak momentum in the non-convex regime and provide a theoretical analysis of momentum's stabilizing effect in the stochastic composite optimization setting, particularly in the small batch regime.
\begin{itemize}[leftmargin=12pt,itemsep=1pt]
    \item First, we establish a lower bound result for the vanilla stochastic proximal gradient method, showing that it cannot converge to the stationary point beyond the variance of the gradient noise.
    \item We study the effect of incorporating momentum into the stochastic proximal gradient method and prove its optimal convergence to the stationary point without any error-dependent mega-batch access. We also rigorously study the variance reduction effect of momentum in the stochastic composite optimization setting.
    \item We further extend our analysis to the case where the proximal steps are solved inexactly and give the same convergence guarantees, demonstrating the robustness of the momentum method.
    \item Finally, we conduct numerical experiments to corroborate our theoretical findings and demonstrate the algorithm's practicality.
\end{itemize}

\section{Related Works}
There is a huge body of work on stochastic and composite optimization. Here, we focus on the stochastic and non-convex regime and do not get into the details of the works in the convex \citep[e.g.][]{ghadimi2012optimal, ghadimi2013optimal} or the deterministic regime \citep[see][]{nesterov2013introductory}. 
The first work that considers the stochastic composite optimization problem in the non-convex case appears to be~\cite{ghadimi2016mini}, in which they established the convergence of the stochastic proximal gradient method to a neighborhood of the stationary point of $F$ up to the variance of the gradient noise. If mega-batches are allowed, they showed that the algorithm takes asymptotically $\cO(\epsilon^{-2})$ stochastic gradient oracle calls to converge to an $\epsilon$ neighborhood of the stationary point of $F$. It was later extended to incorporate the acceleration technique, but the convergence still requires mega-batches. In particular, it requires $k$ samples at the $k$th iteration while the total number of oracle calls is not improved~\cite{ghadimi2013accelerated}. Even though these methods require mega-batches, their upper bounds on the total number of stochastic gradient oracle calls match the lower bound asympotically~\cite{arjevani2023lower}. 

In another research direction, to break the $\cO(\epsilon^{-2})$ lower bound for stochastic non-convex optimization, there have been a long line of works on variance reduction techniques, including Prox-Spiderboost~\cite{wang2019spiderboost} (composite variant of Spider~\cite{fang2018spider}),  Hybrid-SGD~\cite{tran2022hybrid}, and PStorm~\cite{xu2022momentumbased} (composite variant of Storm~\cite{cutkosky2019momentum}). These methods achieve a $\cO(\epsilon^{-\nicefrac{3}{2}})$ asymptotic complexity at the cost of two additional assumptions on the objective and problem structure: \vspace{-1mm} %
\begin{itemize}[leftmargin=12pt,itemsep=1pt,topsep=0pt]
    \item $ f$ is $\wtilde L$-average smooth: $\Eb{\norm{\nabla  f(\xx,\xi) - \nabla  f(\yy,\xi)}^2}\leq \wtilde{L}^2\norm{\xx-\yy}^2$. Note that this assumption might not be satisfied for some simple smooth functions, and it is strictly stronger than the standard smoothness assumption.\footnote{By Jensen's inequality, we have that any $\wtilde L$-average smooth function is also $\wtilde L$-smooth, see also \Cref{assumption:smoothness}. There exists $L$-smooth functions that are not average smooth.}
    \item These variance reduction techniques require the access to two stochastic gradients $\nabla  f(\xx_k,\xi_{k})$ and $\nabla F(\xx_{k-1},\xi_{k})$ at iteration $k$, which might not always be available in practice \citep[e.g.][]{doan2019finite, chen2022efficient}.
\end{itemize}
\paragraph*{Convergence Criteria} The first-order convergence criteria for non-convex optimization problems is an ongoing research topic.  Our work mostly focuses on the convergence in terms of $\norm{\nabla F(\xx_k)}^2$, which we think is the most natural criteria. Many early works study the convergence of problem~\eqref{eq:composite} in terms of the proximal gradient mapping~\cite{ghadimi2016mini, ghadimi2013accelerated, wang2019spiderboost,tran2022hybrid,xu2022momentumbased}. More recently, there have been several works proposing to study the convergence in terms of Moreau envelope and discussing how the different criterias affect the corresponding convergence complexities~\cite{davis2019stochastic,pmlr-v119-zhang20p}. We discuss the differences and connections between these definitions in \Cref{sec:convergence-criteria}.

\paragraph*{Stochastic Polyak Momentum} The idea of using Polyak momentum (a.k.a.\ heavy-ball momentum) was first proposed in~\cite{polyak1964some} for strongly convex quadratic objective in the deterministic case. The first non-asymptotic analysis of \algname{SGD} with Polyak momentum in the smooth non-convex regime is given in~\cite{yu2019linear} and was later refined in~\cite{liu2020improved}. Some recent works also studied how Polyak momentum can be used to remove the dependence on large batches in various settings, e.g.\ for normalized \algname{SGD}~\cite{cutkosky2020momentum} and for communication compressed optimization~\cite{fatkhullin2023momentum}. The analysis of stochastic gradient methods with Polyak momentum remains an active research topic~\cite{wang2020quickly, sebbouh2021almost, li2022last, jelassi2022towards}.

\section{Problem Formulation and Assumptions}

In this work, we consider the composite optimization problem:
\[
    \min_{\xx\in\R^d}\{F(\xx)\eqdef  f(\xx)+\psi(\xx)\} \,.
\]

In the main body of the paper we study the convergence of the algorithms in terms of $\norm{\nabla F(\xx)}^2$ where we assume that $\psi$ is differentiable. We discuss the case where $\psi$ is convex and non-differentiable in \Cref{sec:non-differentiable}. We do not consider the non-convex and non-differentiable case for $\psi$, as this is an exotic scenario in the literature.

Note that most of the existing works on stochastic composite optimization assume that $\psi$ is convex and analyze the convergence of their algorithms in terms of the proximal gradient mapping~\cite{ghadimi2016mini, ghadimi2013accelerated, wang2019spiderboost,tran2022hybrid,xu2022momentumbased}, which is closely related to what we consider in this work. In \Cref{sec:convergence-criteria}, we discuss the proximal gradient mapping and argue that our definition is more natural in the non-convex regime.

We first introduce the following assumption on the monotonicity of the proximal step:

\begin{assumption}
    \label{assumption:psi}
    We assume that, at each iteration $k$ of the algorithms, the output $\xx_{k+1}$ of the proximal step satisfies $\Omega_k(\xx_{k+1}) \leq \Omega_k(\xx_k)$.
\end{assumption}
This is a technical assumption that provides a more nuanced characterization of $\psi$ beyond convexity, but all our theories work with just the simple convexity assumption as well (in fact, several constants can be improved under the convexity assumption). It is natural to assume that the proximal step finds a no worse point than the current iterate, even when $\psi$ is non-convex.

Next, we introduce the lower boundedness assumption of the objective function $F$:

\begin{assumption}
    \label{assumption:finite-min}
    We assume that there exists $F^\star\in \R$ such that $F^\star\leq F(\xx),\forall \xx\in\R^d$.
\end{assumption}

Now, we introduce the smoothness assumption on $ f$. Note that this is a weaker assumption than the average smoothness assumption used in the variance reduction type methods~\cite{ wang2019spiderboost,tran2022hybrid,xu2022momentumbased}.

\begin{assumption}
    \label{assumption:smoothness}
    We assume that the function $ f$ has $L$-Lipschitz gradient, i.e.\ for any $\xx,\yy\in\R^d$, we have:
    \[
        \norm{\nabla  f(\xx)-\nabla  f(\yy)}\leq L\norm{\xx-\yy} \,.
    \]
\end{assumption}

Next, we make an assumption on the noise of the gradient oracle of $ f$.

\begin{assumption}
    \label{assumption:noise}
    We assume that we have access to a gradient oracle $\nabla  f(\xx,\xi)$ for $ f$ such that for all $\xx\in \R^d$ it holds that:
    \begin{equation}
        \label{eq:noise-bound}
        \begin{aligned}
        &\Eb{\nabla  f(\xx,\xi)} = \nabla  f(\xx)\,, \\ &\Eb{\norm{\nabla  f(\xx,\xi)-\nabla  f(\xx)}^2}\leq \sigma^2.
        \end{aligned}
    \end{equation}
\end{assumption}

This is a standard assumption in the analysis of stochastic gradient methods~\cite{nemirovskij1983problem, bubeck2015convex}. Taking mini-batches is equivalent to dividing the variance by the batch size.

\section{Lower Bound for the Vanilla Stochastic Proximal Gradient Method}
\label{sec:lower-bound}
In \Cref{sec:spg} we briefly discussed that the vanilla stochastic proximal gradient method cannot converge to the stationary point of $F$ beyond the variance of the gradient noise. We give a simple example to illustrate this phenomenon and give some intuitions on why momentum resolves the issue.

\begin{restatable}{proposition}{lowerBoundSGD}
    \label{prop:lower-bound-sgd}  
    For any $K\geq 1$ and any (predefined) stepsize coefficients $\{M_k\}_{k=0}^{K - 1}$ (possibly depending on the problem parameters $L,\sigma^2$ and $K$), there exists a problem instance of~\eqref{eq:composite} with $f(\xx)\eqdef \frac{L}{2}\norm{\xx}^2$ and $\psi(\xx)\eqdef \frac{a}{2}\norm{\xx}^2$, $a\eqdef\max_{0 \leq k \leq K - 1} M_k$, satisfying \cref{assumption:finite-min,assumption:smoothness}, and the stochastic gradient oracle $\nabla f(\xx,\xi)\eqdef  \nabla f(\xx)+\xi$, $\xi\sim \cN(0,\sigma^2 \mI)$, satisfying \cref{assumption:noise}, such that, for the sequence $\{\xx_k\}_{k=1}^K$ generated by method~\eqref{eq:proximal-step} started from any initial point~$\xx_0$, it holds that $\Eb{\norm{\nabla F(\xx_k)}^2}\geq \frac{1}{4} \sigma^2$ for any $1 \leq k \leq K$.
\end{restatable}

\begin{proof}
    Clearly, the construction satisfies our assumptions.
    Let us fix an arbitrary $0 \leq k \leq K - 1$.
    Setting the gradient of the auxiliary function in~\eqref{eq:proximal-step} to zero, we see that $\gg_k + a \xx_{k + 1} + M_k (\xx_{k + 1} - \xx_k) = 0$.
    Since $\gg_k = L \xx_k + \xi_k$, it follows that
    \[
        \xx_{k+1} = \frac{ M_k -L}{ M_k+a} \xx_k - \frac{\xi_k}{ M_k+a}.
    \]
    Substituting $\nabla F(\xx_{k + 1}) = (L + a) \xx_{k + 1}$, we obtain
    \begin{align*}
        \Eb{\norm{\nabla F(\xx_{k+1})}^2} & =\frac{(L+a)^2}{(M_k+a)^2}\Eb{\norm{( M_k -L)\xx_k-\xi_k}^2}\\
         &= \frac{(L+a)^2}{(M_k+a)^2}\left(( M_k -L)^2 \Eb{\norm{\xx_k}^2}+\sigma^2\right) \\
        & \geq \frac{(L+a)^2\sigma^2}{(M_k+a)^2}.
    \end{align*}
        
    If $a \geq M_k$, then $\frac{L + a}{M_k + a} \geq \frac{1}{2}$, and the claim follows.
\end{proof}

Note that \Cref{prop:lower-bound-sgd} holds for any initial point $\xx_0$. Even if we start at the optimal point $\xx_0=\0$, the very first step of the method will already incur an $\cO(\sigma^2)$ error and no subsequent steps will be able to reduce it.

It is important that, for the composite optimization problem~\eqref{eq:composite}, we do not make any significant assumptions on~$\psi$.
In particular, we do not assume that $\psi$ is smooth with a certain smoothness constant. \Cref{prop:lower-bound-sgd} demonstrates that, without such extra assumptions, for any fixed choice of parameters for the stochastic proximal gradient method~\eqref{eq:proximal-step}, i.e., the stepsize coefficients~$M_k$ and the number of iterations~$K$, there is always a ``bad'' function in our problem class (namely, $F(\xx)= f(\xx)+\psi(\xx)$ with $f(\xx)=\frac{L}{2}\norm{\xx}^2$ and $\psi(\xx)=\frac{a}{2}\norm{\xx}^2$ for a sufficiently large~$a$) for which the method cannot reach any error $< \frac{1}{4} \sigma^2$ after $K$ steps. In other words, for any given target accuracy $\epsilon < \frac{1}{4} \sigma^2$, it is impossible to find \emph{one specific} choice of the parameters for the method allowing it to reach the $\epsilon$-error on \emph{any} problem from our class.

The problem is that the variance of the gradient noise keeps the iterates away from the stationary point of $F$, and we need some mechanism reducing this variance with time. One such mechanism is the Polyak momentum which we discuss next. 

\section{The Algorithm and Analysis}
\label{sec:main-results}

In this section, we present the stochastic proximal gradient method with the Polyak momentum and establish its convergence guarantees in the non-convex regime. 

\begin{algorithm}[tb]
    \caption{Proximal Gradient Method with Polyak Momentum}
    \label{alg:composite-momentum}
    \begin{algorithmic}[1]
        \State {\bfseries Input:} $\xx_0,\mm_{-1}$ and $\left\{M_k\right\}_{k=0}^\infty,\left\{\gamma_k\right\}_{k=-1}^\infty, \left\{\delta_k\right\}_{k=0}^\infty$, 
        \For{$k = 0,1,2,\dots$}
        \State Compute $\gg_k  = \nabla  f(\xx_k ,\xi_k )$
        \State Update $\mm_{k} = (1-\gamma_{k-1})\mm_{k-1} +\gamma_{k-1}\gg_k $ 
        \State Compute approximate stationary point $\xx_{k+1}$ of
             $\Omega_k(\xx)\eqdef \inp{\mm_{k}}{\xx }+\psi(\xx)+\frac{ M_k }{2}\norm{\xx-\xx_k }^2$
             such that $\Omega_k(\xx_{k+1})\leq \Omega_k(\xx_k) $ and $ \norm{\nabla \Omega_k(\xx_{k+1})}^2\leq\delta_k$
        \EndFor
    \end{algorithmic}	
\end{algorithm}

The method is shown in \Cref{alg:composite-momentum}. In the algorithm, the stochastic gradient $\gg_k$ at each iteration $k$ is replaced by the momentum-aggregated gradient mean:
\[
    \mm_{k} = (1-\gamma_{k-1})\mm_{k-1} +\gamma_{k-1}\gg_k \,.
\]
Our analysis will show that the distance between the momentum and the full gradient decreases as the number of iterations increases, which is the key to resolving the issue of the vanilla stochastic proximal gradient method.

In this section, we assume that the proximal steps are solved exactly, i.e.\ we further make the following assumption:
\begin{assumption}
    \label{assumption:exact-proximal-step}
    At each iteration $k$, we have $\delta_k=0$, i.e. $\nabla \Omega_k(\xx_{k+1})=0$.
\end{assumption}
We relax \Cref{assumption:exact-proximal-step} and discuss the inexact proximal steps in \Cref{sec:inexact-proximal-step}.
\subsection{Convergence Analysis}
We discuss the convergence analysis here, and missing proofs can be found in \Cref{sec:proofs-main-results}.
The convergence analysis of \Cref{alg:composite-momentum} revolves around the following quantities:
\begin{equation}
    \label{eq:FG-def}
    \begin{array}{ccl}
        F_k &\eqdef& \Eb{F(\xx_k)-F^\star} \,,\\
         \Delta_k  &\eqdef& \Eb{\norm{\mm_{k}-\nabla  f(\xx_k )}^2} \,,\\
         R_k &\eqdef &\Eb{\norm{\xx_{k+1}-\xx_k}^2} \,.
    \end{array}
\end{equation}
$F_k$ quantifies the distance between the current objective value and the lower bound. $\Delta_k $ bounds the distance between the current gradient estimate and the full gradient. $R_k$ bounds the distance between two consecutive iterates. We start by giving a descent lemma on $\Delta_k $, which is key to analyzing the variance reduction effect of momentum. Similar statements can be found in~\cite{cutkosky2020momentum} and~\cite{fatkhullin2023momentum}.

\begin{restatable}{lemma}{descentDelta}
    \label{lem:descent-Delta}
    Under \Cref{assumption:noise}, for any $k \geq 0$:
    \[
        \Delta_{k+1}  \leq (1-\gamma_k)\Delta_k +\frac{L^2}{\gamma_k}R_k+\gamma_k^2\sigma^2 \,.
    \]
\end{restatable}
The $(1-\gamma_k)$ factor in the above lemma is the key to show that $\Delta_k $ decreases in the optimization process. Next, we discuss the per-iteration descent of $F_k$. 

Recall that in this section, we have \Cref{assumption:exact-proximal-step}, which is equivalent to the following stationarity assumption:
\begin{equation}
    \label{eq:proxstep-optcond}
    \nabla \psi(\xx_{k+1}) + \mm_{k} +  M_k (\xx_{k+1}-\xx_k) = 0.
\end{equation}

When $\psi$ is convex and non-differentiable, $\nabla \Omega_k(\xx_{k+1})$ is just a subgradient of $\Omega_k$ at $\xx_{k+1}$ that equals to zero, whose existence is guaranteed as the optimality condition. In the following lemmas and theorems, we can simply replace the gradient of $\Omega$ and $\psi$ with such choices of the subgradients and obtain the same results for convex and non-differentiable $\psi$. 

In \Cref{sec:inexact-proximal-step}, we give an approximate stationarity assumption, which is more realistic in practice and shows the same convergence guarantees.

Now we give the following descent lemma on $F_k$:

\begin{restatable}{lemma}{descentF}
    \label{lem:descent-F}
    Under \Cref{assumption:psi,assumption:smoothness,assumption:exact-proximal-step}, for any $k \geq 0$, we have
    \[
        F_{k+1} \leq F_k - \frac{ M_k -L}{4}R_k+\frac{\Delta_k }{ M_k -L}.
    \]
\end{restatable}
\begin{remark}
    The constants in \Cref{lem:descent-F} can be slightly improved if \Cref{assumption:psi} is replaced by the convexity of~$\psi$.
\end{remark}
\begin{proof}
    By \Cref{assumption:smoothness}, we have:
    \begin{align*}
        F(\xx_{k+1}) &=  f(\xx_{k+1}) +\psi(\xx_{k+1})\\
        &\leq  f(\xx_k) + \inp{\nabla  f(\xx_k)}{\xx_{k+1}-\xx_k} \\
        &\quad + \frac{L}{2}\norm{\xx_{k+1}-\xx_k}^2+\psi(\xx_{k+1})\\
        &= f(\xx_k) + \inp{\mm_{k}}{\xx_{k+1}-\xx_k} + \psi(\xx_{k+1}) \\
        &\quad + \frac{ M_k }{2}\norm{\xx_{k+1}-\xx_k}^2 - \frac{ M_k -L}{2}\norm{\xx_{k+1}-\xx_k}^2\\
        &\quad + \inp{\nabla  f(\xx_k)-\mm_{k}}{\xx_{k+1}-\xx_k} \,.
    \end{align*}
    Then we apply \Cref{assumption:psi} and notice that $\Omega_k(\xx_k)=\psi(\xx_k)+\inp{\mm_k}{\xx_k}$:
    \begin{align*}
        F(\xx_{k+1}) &\leq F(\xx_k) - \frac{ M_k -L}{2}\norm{\xx_{k+1}-\xx_k}^2\\
        &\quad + \inp{\nabla  f(\xx_k)-\mm_{k}}{\xx_{k+1}-\xx_k} \,.
    \end{align*}
    Now we apply Young's equality and get
    \begin{align*}
        F(\xx_{k+1}) 
            &\leq F(\xx_k)- \frac{ M_k -L}{4}\norm{\xx_{k+1}-\xx_k}^2 \\
            &\quad + \frac{\norm{\mm_{k}-\nabla  f(\xx_k)}^2}{ M_k -L}\,.
    \end{align*}
    We get the desired result by subtracting $F^\star$ and taking expectation on both sides.
\end{proof}
Note that these will be enough if we aim to prove convergence in terms of the distance between the iterates, but not enough to guarantee a convergence in terms of $\Eb{\norm{\nabla F(\xx_k)}^2}$. Therefore, we relate the distance between iterates and the norm of the gradient in the following lemma:

\begin{restatable}{lemma}{gradfBound}
    \label{lem:gradf-bound}
    Under \cref{assumption:smoothness,assumption:exact-proximal-step}, for any $k \geq 0$,
    \[
        (M_k^2 + L^2) R_k \geq \frac{1}{3} \Eb{\norm{\nabla F(\xx_{k+1})}^2} - \Delta_k.
    \]
\end{restatable}
\begin{proof}
    We can split $\nabla F(\xx_{k+1})$ in the following way:
    \begin{align*}
        \nabla F(\xx_{k+1}) & = \nabla  f(\xx_{k+1})+\nabla\psi(\xx_{k+1})\\
            & = \mm_{k}+\nabla\psi(\xx_{k+1})+(\nabla  f(\xx_k)-\mm_{k})\\
            &\quad +(\nabla  f(\xx_{k+1})-\nabla f(\xx_k)) \,.
    \end{align*}
    Therefore,
    \begin{align*}
        \norm{\nabla F(\xx_{k+1})}^2
        \leq
        3\norm{\mm_{k}+\nabla\psi(\xx_{k+1})}^2
        + 3 \norm{\mm_{k}-\nabla  f(\xx_k)}^2
        + 3 \norm{\nabla  f(\xx_{k+1})-\nabla f(\xx_k)}^2.
    \end{align*}
    Now apply the stationarity condition \eqref{eq:proxstep-optcond} on the first term, and use \Cref{assumption:smoothness} on the third term, we get:
    \[
        \norm{\nabla F(\xx_{k+1})}^2 \leq 3( M_k ^2+L^2)\norm{\xx_{k+1}-\xx_k}^2 + 3\norm{\mm_{k}-\nabla  f(\xx_k)}^2 \,.
    \]
    Rearranging and taking expectations, we get the claim.
\end{proof}

Now, we can piece together all of the above lemmas and consider the following Lyapunov function:
\begin{equation}
    \label{eq:lyapunov-def}
    \Phi_k  \eqdef F_k + a\Delta_k \,,
\end{equation}
where $a$ is a constant to be determined later. Note that $a$ does not have an impact on the algorithm itself, and it only shows up in the analysis. Now we give the following lemma on $\Phi_k $:

\begin{restatable}{lemma}{descentPhi}
    \label{lem:lyapunov} 
    Let \cref{assumption:psi,assumption:smoothness,assumption:noise,assumption:exact-proximal-step} hold,
    and let $a \eqdef \frac{3}{8L}$, $\gamma_k \eqdef \frac{3L}{M_k -L}$ and $ M_k >4L$ for any $k \geq 0$.
    Then, for any $k \geq 0$,
    \begin{equation}
        \label{eq:lyapunov}
        \Phi_{k+1} \leq \Phi_k  - \frac{1}{48 M_k }\Eb{\norm{\nabla F(\xx_{k+1})}^2} + \frac{27L \sigma^2}{4 M_k ^2}.
    \end{equation}
\end{restatable}

We have the following simple corollary for constant stepsize coefficients:
\begin{restatable}{corollary}{convergenceK}
    \label{cor:convergence-K}
    Let \cref{alg:composite-momentum} be run for $K \geq 1$ iterations for solving problem~\eqref{eq:composite} under
    \cref{assumption:smoothness,assumption:noise,assumption:exact-proximal-step},
    with constant coefficients $M_k=M=4L + \frac{3^{\nicefrac{3}{2}}}{2}\sqrt{\frac{KL\sigma^2}{\Phi_0}}$ and $\gamma_k = \frac{3L}{M -L}$ for any $0 \leq k \leq K - 1$, where $\Phi_0 \eqdef F(\xx_0) - F^* + \frac{3}{8 L} \Eb{\norm{\mm_0 - \nabla f(\xx_0)}^2}$.
    Then,
    \begin{equation}
        \label{eq:convergence-K}
        \Eb{\norm{\nabla F(\xx_{t})}^2}
        \leq 
        48(3^{\nicefrac{3}{2}})\sqrt{\frac{L\Phi_0\sigma^2}{K}} + \frac{192L\Phi_0}{K},
    \end{equation}
    where $t$ is chosen uniformly at random from $\{1,\ldots, K\}$.
\end{restatable}

It is also possible to use time-varying coefficients~$M_k$ which do not require fixing the number of iterations in advance.
However, in terms of the convergence rate, it incurs an extra logarithmic factor.
\begin{restatable}{corollary}{convergenceKIncreasingMk}
    \label{cor:convergence-K-increasing-Mk}
    Consider \cref{alg:composite-momentum} for solving problem~\eqref{eq:composite} under
    \cref{assumption:smoothness,assumption:noise,assumption:exact-proximal-step}
    with coefficients $M_k=\max\bigl\{\sqrt{\frac{(k+1)L\sigma^2}{\Phi_0}},4L\bigr\}$, $\gamma_k = \frac{3L}{M_k -L}$
    for any $k \geq 0$, where $\Phi_0 \eqdef F(\xx_0) - F^* + \frac{3}{8 L} \Eb{\norm{\mm_0 - \nabla f(\xx_0)}^2}$.
    Then, for any $k \geq 1$, we have
    \begin{equation}
        \label{eq:convergence-K-increasing-Mk}
        \hspace{-0.2em}
        \Eb{\norm{\nabla F(\xx_{t(k)})}^2}
        \leq
        372 \ln(e k) \Bigl( \sqrt{\frac{L\Phi_0\sigma^2}{k}} + \frac{L\Phi_0}{k} \Bigr),
    \end{equation}
    where $t(k)$ is chosen randomly from $\{1,\ldots, k\}$ with probabilities $\Pr(t(k) = i) \propto \frac{1}{M_{i - 1}}$, $i = 1, \ldots, k$, and $e \eqdef \exp(1)$.
\end{restatable}

Let us point out that using a random iterate as the output of the algorithm is standard in the literature (see, e.g., \cite{rakhlin2012making}) and can be efficiently implemented without fixing the number of iterations in advance. We discuss this more carefully in \Cref{sec:sampling}.
\subsection{Initialization and Convergence Guarantees}

As we can see from \cref{cor:convergence-K,cor:convergence-K-increasing-Mk}, the convergence rate of \cref{alg:composite-momentum} depends on~$\Phi_0 \eqdef \cO(F_0+\frac{1}{L} \Delta_0)$, where $\Delta_0 \eqdef \Eb{\norm{\mm_0 - \nabla f(\xx_0)}^2} = \Eb{\norm{(1 - \gamma_{-1}) \mm_{-1} + \gamma_{-1} \gg_0 - \nabla f(\xx_0)}^2}$ depends on~$\mm_{-1}$.
There are subtle differences in how $\mm_{-1}$ can be initialized between the non-composite and composite cases.
\paragraph*{Non-Composite Case:} When $\psi\equiv 0$, i.e. $F\equiv  f$, we set the initial momentum $\mm_{-1}\eqdef -\frac{\gamma_{-1}}{1-\gamma_{-1}} \gg_0$, then $\Delta_0=\Eb{\norm{\mm_0-\nabla  f(\xx_0)}^2}=\norm{\nabla  f(\xx_0)}^2 \leq 2 L F_0$, where the last inequality follows from \Cref{assumption:smoothness} and the fact that $F\equiv  f$. Therefore, in the non-composite case, we can initialize the parameters such that $\Phi_0=\cO(F_0)$.
\paragraph*{Composite Case:} The composite case is slightly trickier. We set $\mm_{-1}\eqdef \gg_0$ and get that $\mm_0=\gg_0$. Hence $\Delta_0=\Eb{\norm{\gg_0-\nabla f(\xx_0)}^2}\leq \sigma^2$. Therefore, $\Phi_0=\cO(F_0+\frac{\sigma^2}{L})$. When $\sigma^2=\cO(LF_0)$, we get the same $\Phi_0=\cO(F_0)$ as in the non-composite case. 
\paragraph*{Mini-Batch Initialization:} In the case that $\sigma^2$ is much larger than $F_0$, if we have access to a constant size (not depending on the target error) mini-batch initially, then we can set $\gg_0=\frac{1}{b_0}\sum_{i=1}^{b_0}\nabla  f(\xx_0,\xi_i)$ where $b_0\eqdef \ceil{\frac{\sigma^2}{LF_0}}$, i.e.\ $\gg_0$ is a mini-batch stochastic gradient of size $b_0$. Then we have $\Phi_0=\cO(F_0+\frac{\sigma^2}{b_0L})=\cO(F_0)$, which is the same as in the non-composite case. 

We have thus proved the following convergence guarantee for \Cref{alg:composite-momentum}. 
\begin{restatable}{theorem}{convergenceExact}
    \label{thm:convergence-exact}
    Consider \cref{alg:composite-momentum}, as applied to solving problem~\eqref{eq:composite} under \cref{assumption:smoothness,assumption:noise,assumption:exact-proximal-step}, run for
    $K = \cO\bigl(\frac{L\Phi_0\sigma^2}{\varepsilon^2} + \frac{L\Phi_0}{\varepsilon}\bigr)$ iterations with constant coefficients $M_k = M = 4L + \frac{3^{\nicefrac{3}{2}}}{2}\sqrt{\frac{KL\sigma^2}{\Phi_0}}$ and $\gamma_k = \frac{3L}{ M -L}$ for any $0 \leq k \leq K - 1$,
    where $\Phi_0 \eqdef F_0 + \frac{3}{8 L} \Eb{\norm{\mm_0 - \nabla f(\xx_0)}^2}$, $F_0 \eqdef F(\xx_0) - F^*$ and $\varepsilon > 0$ is a given target error.
    Then, for the point $\xx_t$ chosen uniformly at random from $\{\xx_1,\ldots ,\xx_K\}$ it holds that $\Eb{\norm{\nabla F(\xx_t)}^2}\leq \epsilon$.

    If $\psi\equiv 0$, we can initialize $\mm_{-1}$ in such a way that $K=\cO\bigl(\frac{LF_0\sigma^2}{\varepsilon^2}+\frac{LF_0}{\varepsilon}\bigr)$. Otherwise, we can initialize $\mm_{-1}$ in such a way that $K=\cO\bigl(\frac{LF_0\sigma^2}{\varepsilon^2} + \frac{\sigma^4}{\varepsilon^2}+\frac{LF_0}{\varepsilon} + \frac{\sigma^2}{\varepsilon}\bigr)$.
    Further, when the initial mini-batch of size $\ceil{\frac{\sigma^2}{LF_0}}$ is allowed, we can initialize $\mm_{-1}$ in such a way that $K=\cO\bigl(\frac{LF_0\sigma^2}{\varepsilon^2}+\frac{LF_0}{\varepsilon}\bigr)$.
\end{restatable}
When an initial mini-batch is not allowed, the convergence rate has an extra $\cO(\frac{\sigma^4}{\epsilon^2}+\frac{\sigma^2}{\epsilon})$ term that is not present in the non-composite case. One natural question is whether this is an artifact of our analysis or inherent to the problem and the algorithm. Note that the first step of the \Cref{alg:composite-momentum} coincides with the vanilla stochastic proximal gradient method. Consider the lower bound construction in \Cref{prop:lower-bound-sgd}: when starting at $\xx_0=0$, which is a stationary point such that $F_0=0$, one such step would incur an $\cO(\sigma^2)$ error. Therefore, in the non-composite case without the initial mini-batch, the convergence rate must have some term that only depends on $\sigma^2$. In other words, the extra terms seem unavoidable.

\section{Variance Reduction Effect of Momentum}
\label{sec:vr-effect}

In \Cref{sec:lower-bound}, we demonstrated that the vanilla stochastic proximal gradient method cannot converge because the gradient noise variance keeps the iterates away from the stationary point of $F$. 
In this section, we show that the momentum term in \Cref{alg:composite-momentum} can reduce the variance of the gradient noise at the same rate as the gradient norm. This is the key to the convergence of \Cref{alg:composite-momentum} without batches. This variance reduction effect has been known in practice and implicitly used in the analysis of~\cite{cutkosky2020momentum} and~\cite{fatkhullin2023momentum}. We precisely characterize such an effect in the composite optimization setting. Proofs are deferred to \Cref{sec:proofs-vr-effect}.

We begin by refining the result of \Cref{lem:lyapunov}:
\begin{restatable}{lemma}{descentPhiRefined}
    \label{lem:lyapunov-refined}
    Let \Cref{assumption:psi,assumption:smoothness,assumption:noise,assumption:exact-proximal-step} hold, and let $a\eqdef \frac{\sqrt{2}}{8L}$, $\gamma_k\eqdef \frac{3\sqrt{2}L}{M_k -L}$ and $M_k>(1+3\sqrt{2})L$ for any $k\geq0$. Then, for any $k\geq 0$,
    \[
        \label{eq:lyapunov-refined}
        \Phi_{k+1} \leq\Phi_k  - \frac{1}{48 M_k }\Eb{\norm{\nabla F(\xx_{k+1})}^2} - \frac{3\sqrt{2}}{2 M_k }\Delta_k + \frac{27\sqrt{2}L}{4 M_k ^2}\sigma^2.
    \]
\end{restatable}
With \Cref{lem:lyapunov-refined}, we can now precisely quantify the variance reduction effect of the momentum method.

\begin{restatable}{theorem}{vrMomentum}
    \label{thm:vr-momentum}
    Consider \cref{alg:composite-momentum}, as applied to solving problem~\eqref{eq:composite} under \cref{assumption:smoothness,assumption:noise,assumption:exact-proximal-step}, run for
    $K = \cO\bigl(\frac{L\Phi_0\sigma^2}{\varepsilon^2} + \frac{L\Phi_0}{\varepsilon}\bigr)$ iterations with constant coefficients $M_k = M = (1+3\sqrt{2})L + \frac{3^{3 / 2}}{2^{3 / 4}} \sqrt{\frac{KL\sigma^2}{\Phi_0}}$ and $\gamma_k = \frac{3\sqrt{2}L}{M -L}$ for any $0 \leq k \leq K - 1$,
    where $\Phi_0 \eqdef F_0 + \frac{\sqrt{2}}{8 L} \Eb{\norm{\mm_0 - \nabla f(\xx_0)}^2}$, $F_0 \eqdef F(\xx_0) - F^*$ and $\varepsilon > 0$ is a given target error.
    Then, for the point $\xx_t$ chosen uniformly at random from $\{\xx_1,\ldots ,\xx_K\}$ it holds that $\Eb{\norm{\mm_{t}-\nabla  f(\xx_t)}^2}\leq \epsilon$.
\end{restatable}
In words, the squared distance between the momentum $\mm_k$ and the full gradient $\nabla f(\xx_k)$ decreases at the same rate as the squared norm of the gradient $\nabla F(\xx_k)$.

\section{Inexact Proximal Step}
\label{sec:inexact-proximal-step}
While many existing works (e.g.~\cite{ghadimi2016mini, ghadimi2013accelerated, wang2019spiderboost,hendrikx2020statistically, tran2022hybrid,xu2022momentumbased}) rely on the assumption that the proximal step can be solved exactly, this might not always be practically possible~\cite{barre2023principled}. In this section, we briefly discuss extending our analysis to the case where the proximal step is solved inexactly and give an inexactness criterion similar to that of~\cite{woodworth2023losses}.

A crucial element in our previous analysis is the assumption that $\nabla \Omega_k(\xx_{k+1})=0$, i.e. \Cref{eq:proxstep-optcond}. Therefore, defining the inexactness criteria as an approximate stationarity condition of $\Omega_k$ at $\xx_{k+1}$ where $\delta_k\neq 0$ is natural. In particular, we define the criteria as follows:
\begin{equation}
    \label{eq:inexactness-criteria}
    \Eb{\norm{\nabla \Omega_k(\xx_{k+1})}^2} \leq \frac{M_k^2}{16} \Eb{\norm{\xx_{k+1}-\xx_k}^2} + S_k,
\end{equation}
where $S_k$ will be decided later. Again, it should be understood that in the non-differentiable case $\nabla \Omega_k(\xx_{k+1})$ is a certain subgradient of $\Omega_k$ at $\xx_{k+1}$ such that \Cref{eq:inexactness-criteria} holds.

Now we state the convergence result (proofs are deferred to \Cref{sec:proofs-inexact}):

\begin{restatable}{theorem}{convergenceInexact}
    \label{thm:convergence-inexact}
    Consider \Cref{alg:composite-momentum}, as applied to solving problem~\eqref{eq:composite} under \Cref{assumption:psi,assumption:smoothness,assumption:noise}, and the approximate stationarity condition at each iteration $k$: $\Eb{\norm{\nabla \Omega_k(\xx_{k+1})}^2} \leq \frac{ M ^2}{16} \Eb{\norm{\xx_{k+1}-\xx_k}^2} + S_k,$, run for $K=\cO\left(\frac{L\Phi_0\sigma^2}{\varepsilon^2} + \frac{L\Phi_0}{\varepsilon}\right)$ iterations with constant coefficients $M_k=M=4L + \sqrt{\frac{8KL\sigma^2}{\Phi_0}}$ and $\gamma_k=\sqrt{\frac{152}{17}}\frac{L}{ M -L}$ for any $0\leq k\leq K-1$, where $\Phi+0 = F_0+\sqrt{\frac{19}{156}}\frac{\Eb{\norm{\mm_0-\nabla f(\xx_0)}^2}}{L}, F_0\eqdef F(\xx_0)-F^{\star}$ and $\varepsilon>0$ is a given target error. Then, for the point $\xx-t$ chosen uniformly at random from $\{\xx_1,\ldots,\xx_K\}$ it holds that $\Eb{\norm{\nabla F(\xx_t)}^2}\leq \frac{\epsilon}{2} + \frac{8}{K}\sum_{k=0}^{K-1}S_k$. In particular, if for any $0\leq k\leq K-1$, $S_k\leq \frac{\epsilon}{16}$, then $\Eb{\norm{\nabla F(\xx_t)}^2}\leq \varepsilon$.
\end{restatable}

It remains to discuss how to minimize $\Omega_k$ such that the inexactness criteria~\eqref{eq:inexactness-criteria} is satisfied. There is a line of research on minimizing gradient norm with \algname{SGD} variants, e.g.\ \cite{allen2018make}, and one can also exploit the structure information in $\psi$ if there is any. Here, we state that \algname{SGD} suffices for our purpose and give the following proposition, which is a simple modification of Proposition 2.6 in~\cite{woodworth2023losses}:
\begin{restatable}{proposition}{proxstepSGD}
    \label{prop:proxstep-sgd}
    If $\psi$ is convex and $L_{\psi}$-smooth, and we have access to an unbiased gradient oracle of $\psi$ with variance at most $\sigma_{\psi}^2$, then after at most 
    \[
        T = \cO\biggl(\frac{L_{\psi} + M}{ M } \ln \frac{L_{\psi} + M}{M} + \frac{(L_{\psi} +  M )\sigma_{\psi}^2}{ M  S_k}\biggr)
    \]
    iterations of \algname{SGD}, the output $\hat \xx$ of \algname{SGD} satisfies the condition~\eqref{eq:inexactness-criteria}.
\end{restatable}
The proof of \Cref{prop:proxstep-sgd} is simple, and we refer interested readers to~\cite{woodworth2023losses} for the discussions therein.

\begin{figure*}[h]
    \centering
    \begin{tabular}{ccc}
        \includegraphics[width=0.3\textwidth]{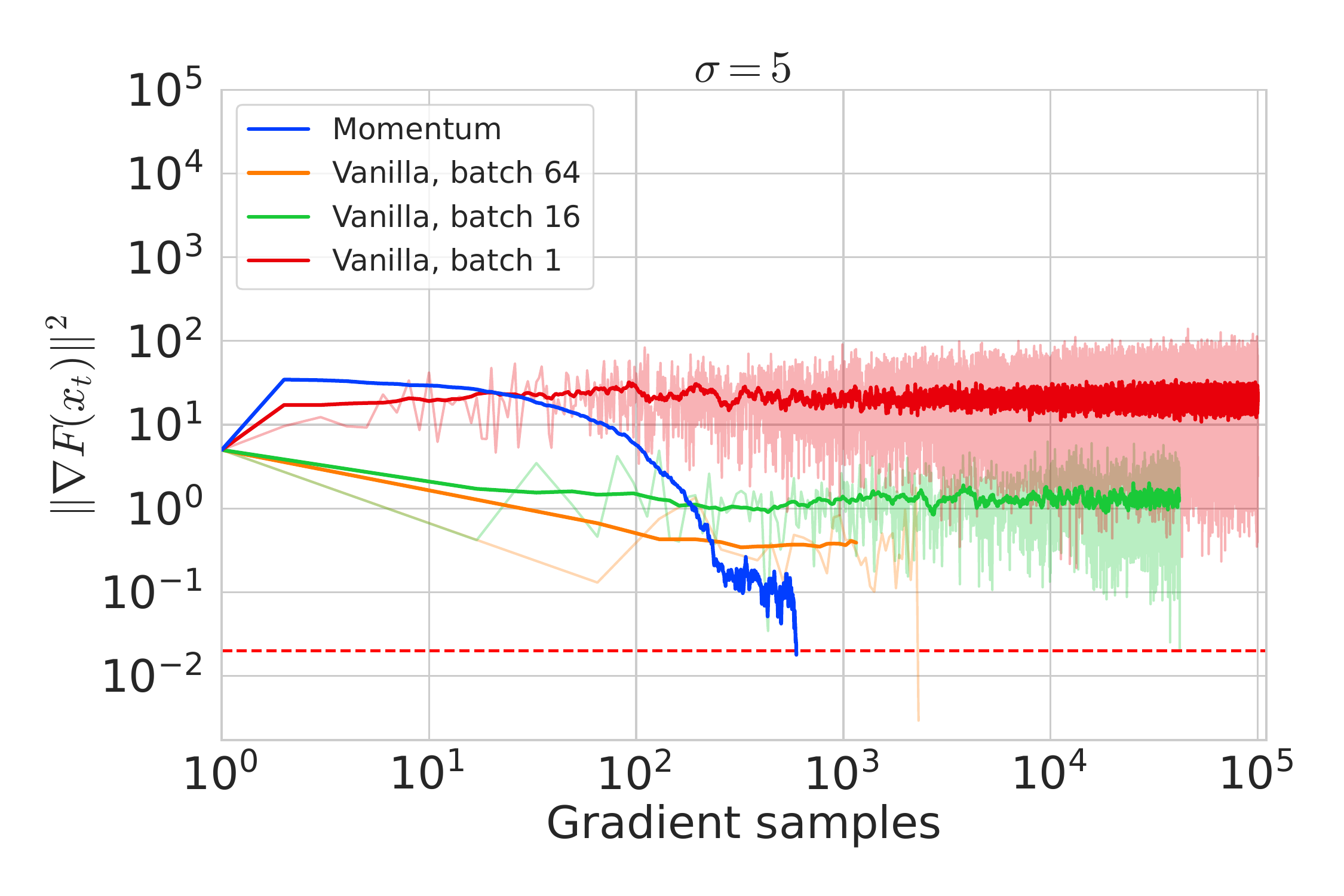}
        & \includegraphics[width=0.3\textwidth]{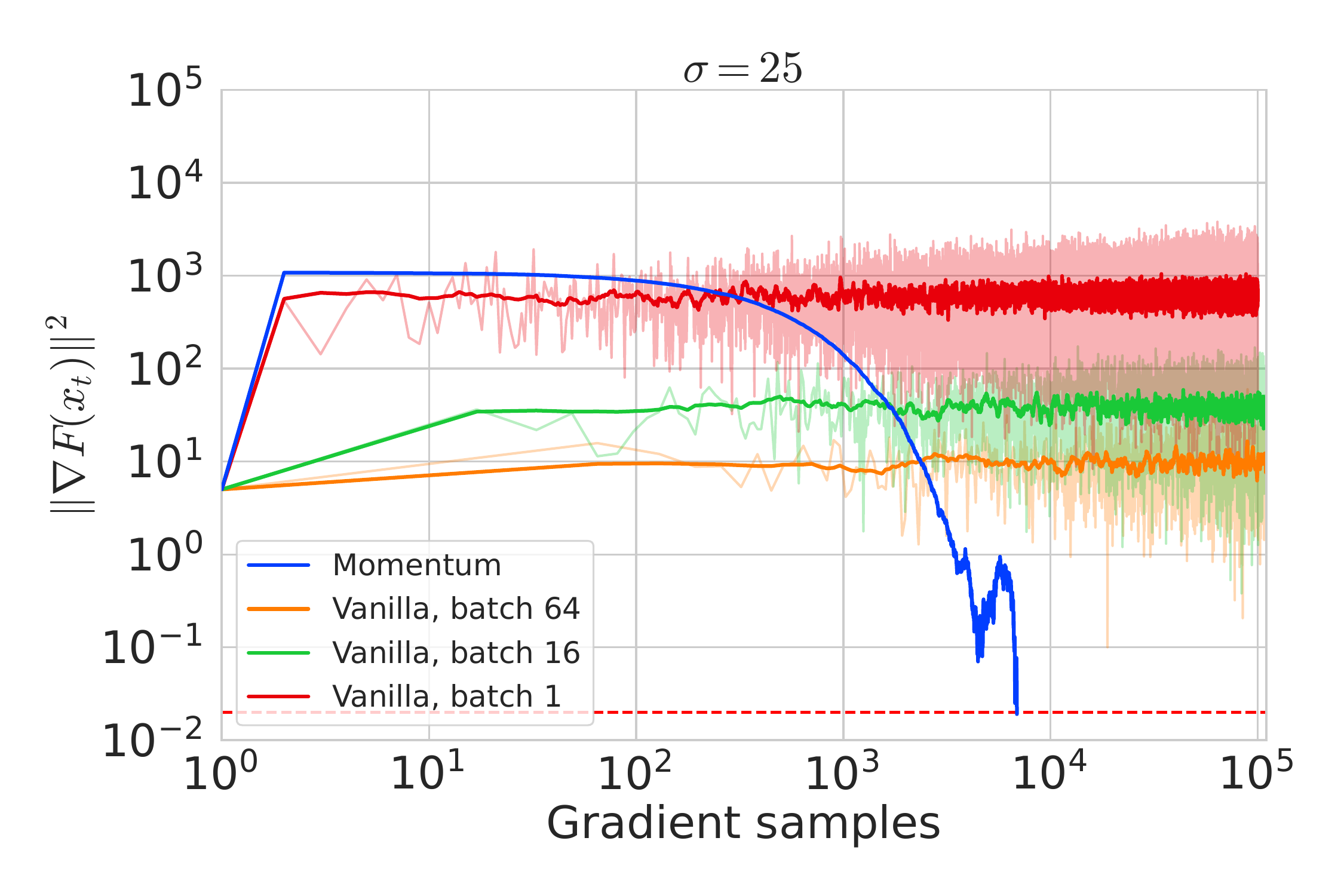}
        & \includegraphics[width=0.3\textwidth]{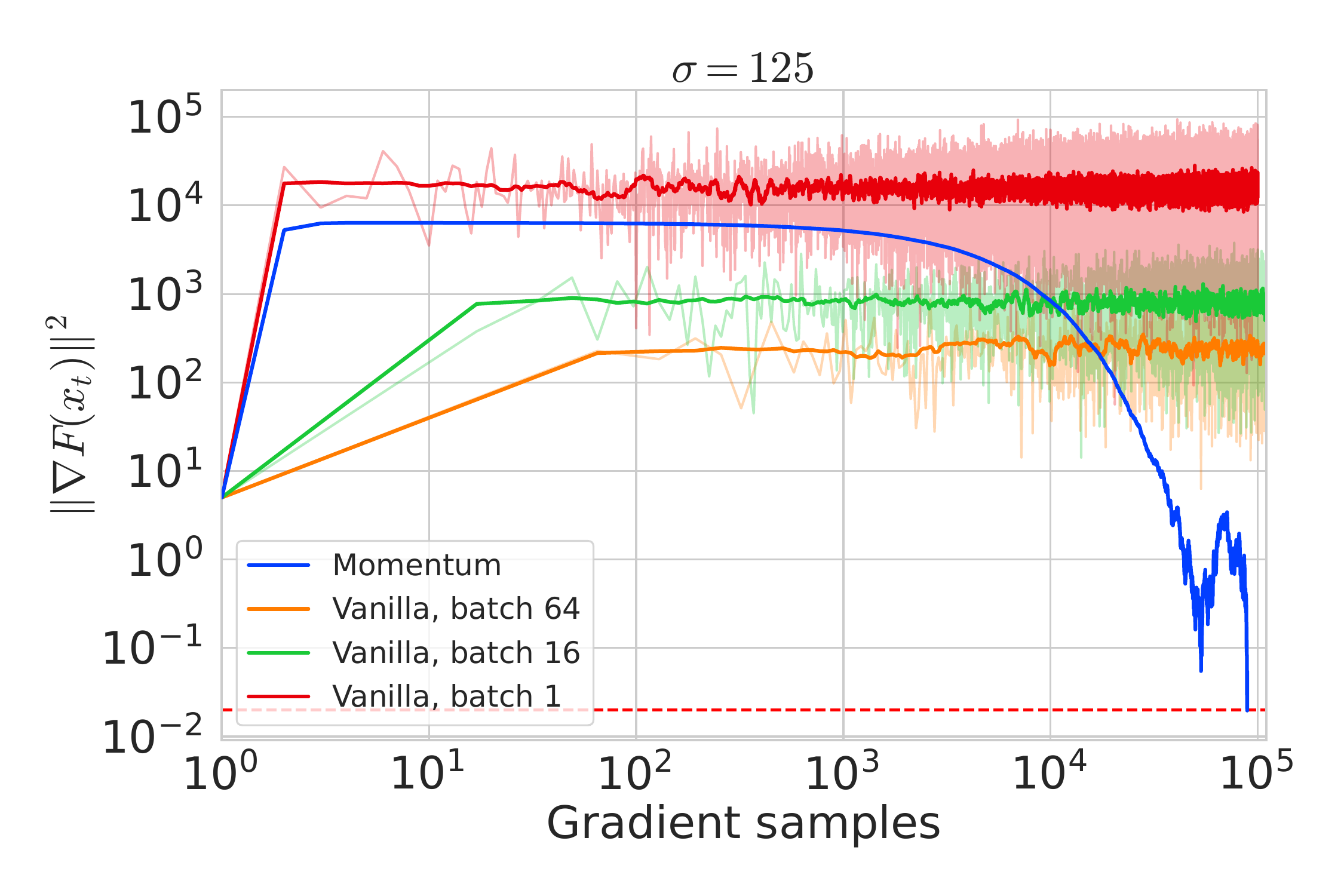}
    \end{tabular}
    \caption{Comparison of \Cref{alg:composite-momentum} and the vanilla stochastic proximal gradient method on the synthetic quadratic problem. For the vanilla stohastic proximal methods, we also highlight the smoothed curves on top of the original curves that oscillate much more. The left, middle, and right figures correspond to $\sigma=5,25,125$, respectively. The vanilla stochastic proximal gradient method uses batch sizes $1, 16, 64$. The x-axis represents the number of gradient samples and is truncated to only show the first $10^5$ gradient samples.}
    \label{fig:synthetic}
\end{figure*}

\section{Experiments}

In this section, we conduct numerical experiments to corroborate our theoretical findings and demonstrate the practical effectiveness of \Cref{alg:composite-momentum}. 

\subsection{Synthetic Quadratic Problem}

We first consider a synthetic quadratic problem inspired by our lower bound construction in \Cref{prop:lower-bound-sgd}. We consider the following problem in $\R^d$: we set $f(\xx)=\frac{L\norm{\xx}^2}{2}$ and $\psi(\xx) = \frac{a\norm{\xx}^2}{2}$, where $a$ is sufficiently large. Note that the vanilla stochastic proximal gradient method and \Cref{alg:composite-momentum} are oblivious to the parameter $a$. We add Gaussian noise to the gradients to control the stochastic level $\sigma^2$ of the stochastic oracle. We simulate the batch sample by dividing the variance by the batch size.

In \Cref{fig:synthetic}, we compare the performance of \Cref{alg:composite-momentum} and the vanilla stochastic proximal gradient method. We set $d=5, L=1$ and $a=10^4$. We run the vanilla method with batch sizes $1, 16, 64$. We set $\sigma=5,25,125$ respectively. The parameter $ M $ is tuned by a grid search in $\{10^0, 10^1, 10^2, 10^3, 10^4\}$ for all methods, and the momentum parameter $\gamma$ is tuned by a grid search in $\{10^{-1}, 10^{-2}, 10^{-3}, 10^{-4}, 10^{-5}\}$. 
We set the maximum number of iterations to be $10^4$, and the tolerance is $0.02$. We see that as $\sigma^2$ increases, \Cref{alg:composite-momentum} still reaches the desired tolerance, while the vanilla method with batch size $1$ fails to converge in all cases, and with batch sizes $16$ and $64$ only converges when $\sigma=5$. In particular, with batch size $1$, the error of vanilla method oscillates around $22$ with $\sigma^2=5^2$, around $672$ with $\sigma^2=25^2$, and around $14381$ with $\sigma^2=125^2$. In other words, the error of the vanilla method is indeed proportional to $\sigma^2$, as predicted by our lower bound result in \Cref{prop:lower-bound-sgd}. We also have that for \Cref{alg:composite-momentum}, $M$ is set to be $10,10,100$ and $\gamma$ is set to be $0.01,0.001,0.0001$ for $\sigma=5,25,125$ respectively, which is consistent with our theoretical prediction that $M$ should increase while $\gamma$ should decrease as $\sigma$ increases.
We also point out that the momentum method exhibits a jump in the error in the first several iterations, consistent with our analysis that the first step of the momentum method incurs an $\cO(\sigma^2)$ error as well.

\subsection{Regularized Machine Learning Experiment}

Now we consider the classical application of composite optimization: regularized machine learning~\cite{Liu_2015_CVPR}. We use the $\ell_{\infty,1}$ regularizer to regularize the weights on each layer. The proximal step with the $\ell_{\infty,1}$ regularizer is implemented in~\cite{murray2019autosizing}. We evaluate the performances of \Cref{alg:composite-momentum} and the vanilla stochastic proximal gradient method on the Cifar-10 dataset~\cite{krizhevsky2014cifar10} with the Resnet-18~\cite{he2016deep}. The regularization parameter is set to be $0.1$, which is observed in~\cite{murray2019autosizing} to achive a balance between enforcing sparsity in the model and maintaining the model performance. We use a batch size of $256$ and run $300$ epochs. We use the standard step size parameter $M=10$ (corresponding to a learning rate of $0.1$) for the experiment. We apply a multi-step learning rate scheduler at $150$ epoch and $250$ epoch, with a decay factor of $0.1$. For \Cref{alg:composite-momentum}, momentum parameter $\gamma $ is set to be $0.1$ by a grid search.

\begin{figure}[H]
    \centering
    \includegraphics[width=0.49\linewidth]{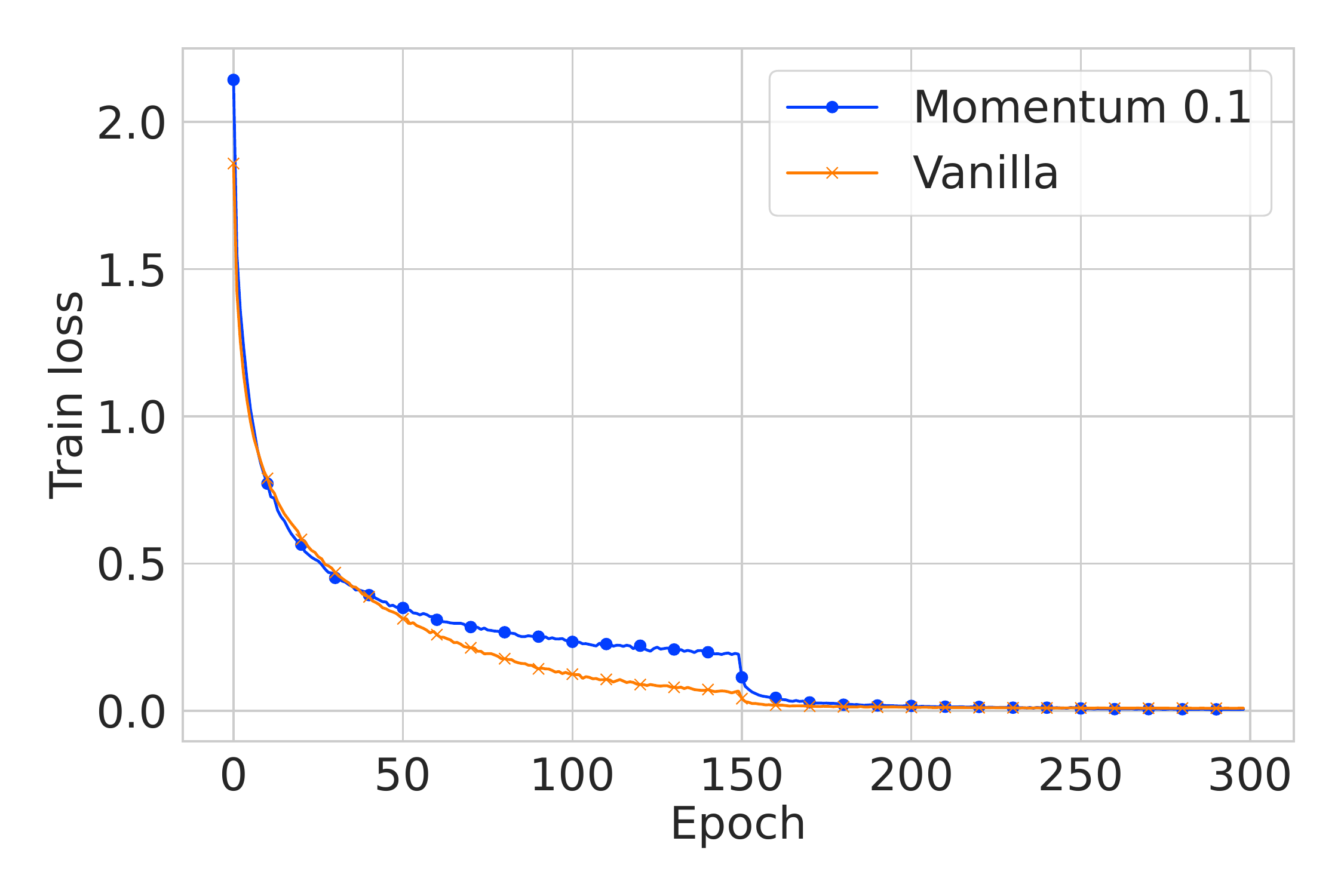}
    \hfill \includegraphics[width=0.49\linewidth]{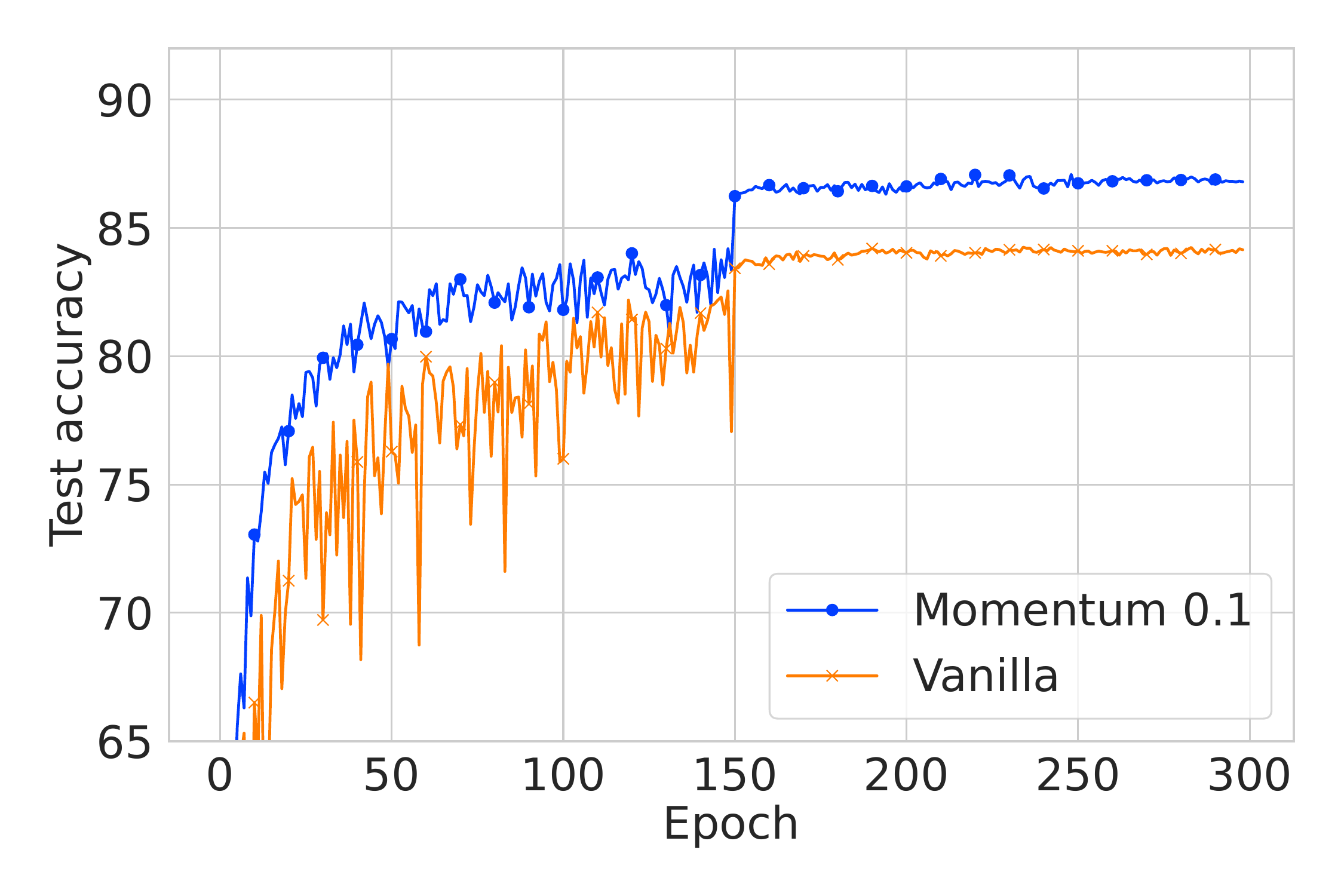}
    \label{fig:regcifar10}
    \caption{Comparison of \Cref{alg:composite-momentum} and the vanilla stochastic proximal gradient method for the $\ell_{\infty, 1}$ regularized machine learning problem on Cifar-10 dataset, with Resnet-18. The left and right figures correspond to the training loss and test accuracy, respectively. }
    \vspace{-2mm}
\end{figure}

We summarize the performances of the two methods in \Cref{fig:regcifar10}. We see that \Cref{alg:composite-momentum} outperforms the vanilla method in terms of both training loss and test accuracy. In terms of the test accuracy, \Cref{alg:composite-momentum} displays a much smoother curve than that of the vanilla method.

In connection to our theoretical observation that the step size parameter $M$ should increase while the momentum parameter $\gamma$ should decrease as the stochastic level $\sigma$ increases, we observe that, grid searches with respect to the train loss lead to the choices $M=100, 100, 10$ and $\gamma = 0.1, 0.1, 0.1$ for batch sizes $64, 128$ and $256$. It appears that the momentum parameter $\gamma$ is less sensitive to the batch size than the step size parameter $M$, and setting $\gamma=0.1$ might be a good choice in practice. We note that while there is certainly a correlation between the batch size and the stochastic level $\sigma^2$, we do not have a direct control over $\sigma^2$, as compared to the synthetic problem.

\section{Conclusion}
In this work we revisit the non-convex stochastic composite optimization problem, and address its convergence issue in the small batch regime. We show that the vanilla stochastic proximal method cannot converge to the stationary point beyond the variance of the gradient noise. We analyze the immensely successful Polyak momentum method in this context and establish its optimal convergence rate without any batch size requirement, demonstrating its superiority over the vanilla method. We conduct numerical experiments to corroborate our theoretical findings. 
In light of the past successes of proximal methods in ML, and the recent emerging application scenarios for proximal methods in DL our findings reinforce the robustness and the potential of the Polyak momentum method. %

\section*{Acknowledgments}
We acknowledge partial funding from Helmholtz AI (project Opt4Bio).

\bibliography{reference}
\bibliographystyle{plainnat}

\newpage

\appendix
\onecolumn

\section{Additional Experiments}
\label{sec:additional-experiments}

In this section we consider the recent statistical preconditioner (proxy training) technique of~\cite{hendrikx2020statistically, woodworth2023losses}, where for some objective $\ell$, we consider $f=\ell-\hat \ell$ and $\psi=\hat \ell$. $\hat \ell$ is a ``statistical preconditioner'' defined on a sub-sample of the training dataset. Here we simulate the setup on Cifar-10 dataset~\cite{krizhevsky2014cifar10}. $\ell$ is defined on the whole $50000$ training images, and $\hat \ell$ is defined on a subset of $2560$ training images. We follow the implementation of~\cite{woodworth2023losses}, with one difference: at each iteration $k$, \cite{woodworth2023losses} computes the full gradient $\nabla \hat \ell(\xx_k)$ while we only compute a stochastic gradient of batch size $128$. In the experiment, we use batch size of $512$ for $\ell$, and a batch size of $128$ for the \algname{SGD} updates on $\Omega_k$. We perform a grid-search on the parameters $M$ and $\gamma$. The \algname{SGD} on $\Omega_k$ takes $20$ iterations and a step-size $0.01$, which is tuned in~\cite{woodworth2023losses}. The experiments demonstrate that \algname{SGD} is effective and reliable for solving the proximal step with sufficient accuracy (see also the experiments in~\cite{woodworth2023losses}). We see that the momentum method outperforms the vanilla method, and the convergence of the momentum method seems smoother than the vanilla method.

\begin{figure}[tb]
    \centering
    \begin{tabular}{cc}
        \includegraphics[width=0.3\linewidth]{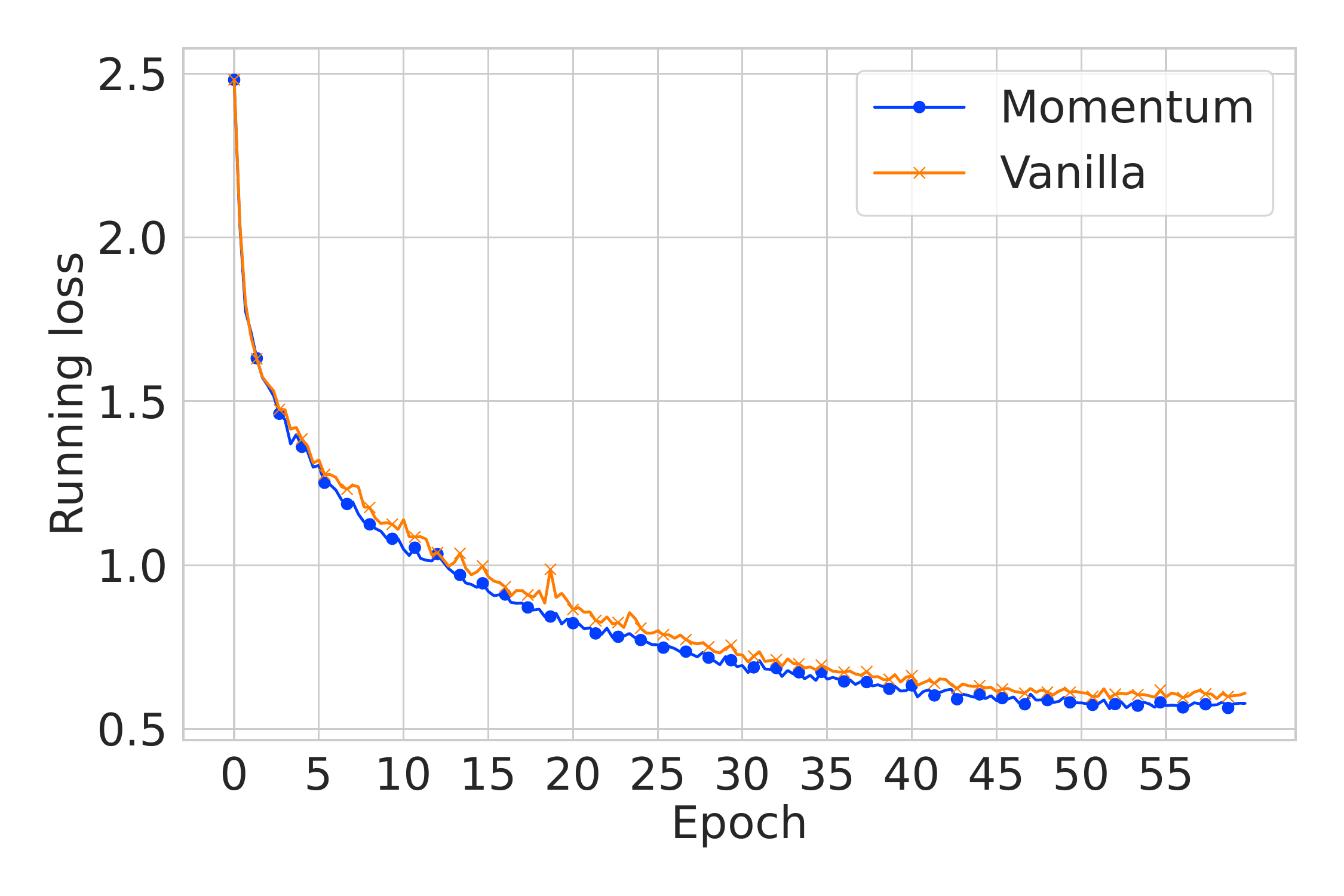} &
        \includegraphics[width=0.3\linewidth]{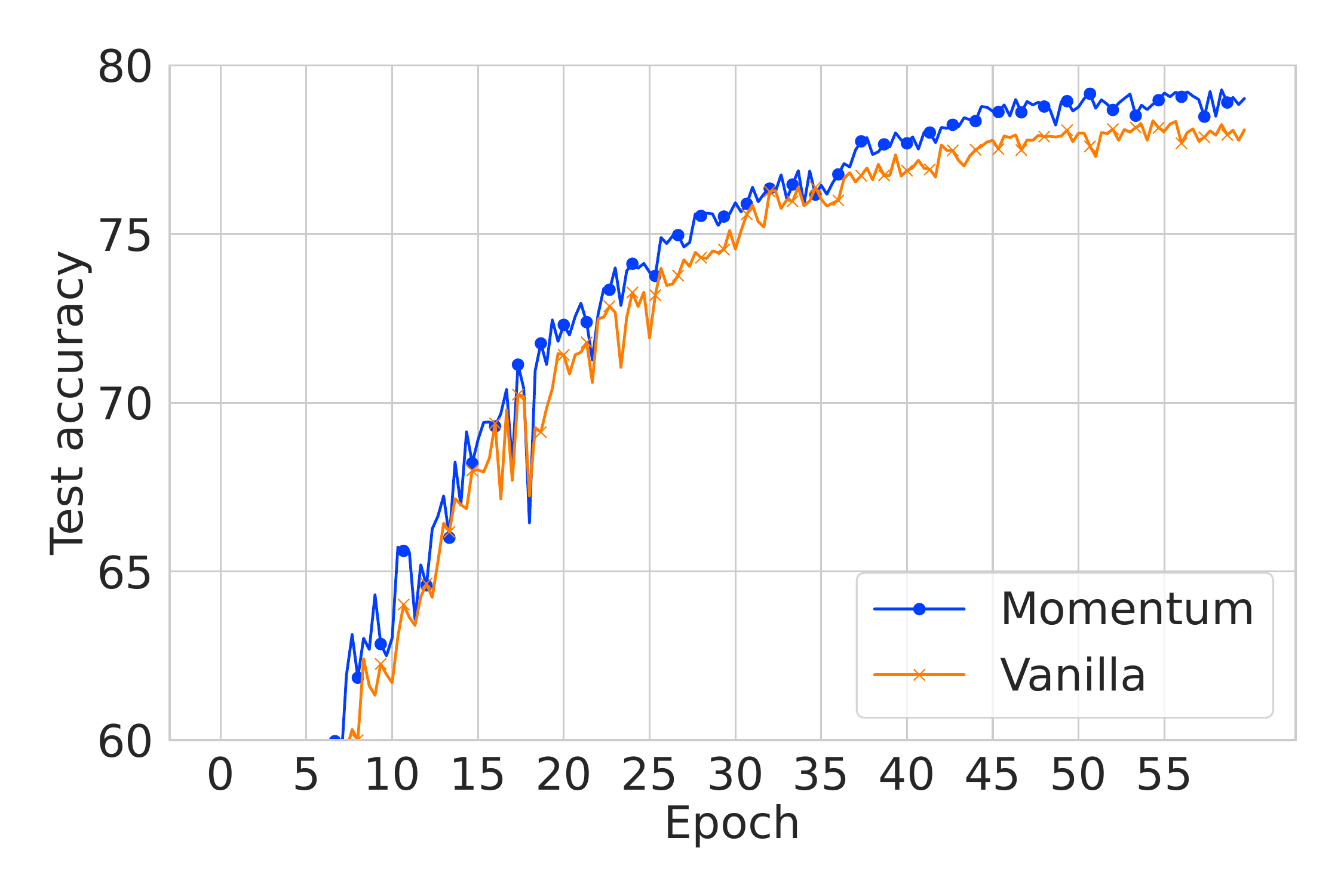}
    \end{tabular}
    
    \label{fig:proxyprox}
    \caption{Comparison of \Cref{alg:composite-momentum} and the vanilla stochastic proximal gradient method for the statistical preconditioning technique on Cifar-10 dataset. The left and right figures correspond to the training loss and test accuracy, respectively. }
    \vspace{-2mm}
\end{figure}

\section{Missing Proofs in Section~\ref{sec:main-results}}
\label{sec:proofs-main-results}
We start by giving the proof of \Cref{lem:descent-Delta}:
\descentDelta*
\begin{proof}
    Indeed,
    \begin{align*}
        \Delta_{k + 1}
        &=
        \Eb{\norm{\mm_{k+1}-\nabla  f(\xx_{k+1} )}^2}
        =
        \Eb{\norm{(1-\gamma_k)(\mm_{k}-\nabla  f(\xx_{k+1}))+\gamma_k(\gg_{k+1} -\nabla  f(\xx_{k+1} ))}^2 }\\
            &= (1-\gamma_k)^2\Eb{\norm{\mm_{k}-\nabla  f(\xx_{k+1} )}^2}+\Eb{\gamma_k^2\norm{\gg_{k+1} -\nabla f(\xx_{k+1} )}^2}\\
            &\leq (1-\gamma_k)^2\Eb{\norm{\mm_{k} -\nabla f(\xx_{k})+\nabla  f(\xx_{k})-\nabla f(\xx_{k+1} )}^2}+\gamma_k^2\sigma^2.
    \end{align*}
    For the second identity, we have used the fact that $\gg_{k+1}$ is unbiased. By Young's inequality, for any $\alpha>0$, we have:
    \begin{align*}
        \Delta_{k + 1}
        &\leq
        (1-\gamma_k)^2(1+\alpha) \Delta_k
        +
        (1-\gamma_k)^2(1+\alpha^{-1})\Eb{\norm{\nabla  f(\xx_{k})-\nabla f(\xx_{k+1})}^2}+\gamma_k^2\sigma^2\\
        &\leq
        (1-\gamma_k)^2(1+\alpha) \Delta_k
        +
        (1-\gamma_k)^2(1+\alpha^{-1})L^2 R_k + \gamma_k^2\sigma^2,
    \end{align*}
    where the last inequality follows from the smoothness of $ f$.
    Choosing now $\alpha\eqdef \frac{\gamma_k}{1 - \gamma_k}$, we get the claimed inequality.
\end{proof}

Now we give the proof of \Cref{lem:lyapunov}:

\descentPhi*
\begin{proof}
    We put together \Cref{lem:descent-Delta,lem:descent-F} and get:
    \begin{align*}
        \Phi_{k+1}
        &=
        F_{k+1} + a\Delta_{k+1}
        \leq
        F_k - \frac{ M_k -L}{4}R_k+\frac{\Delta_k }{ M_k -L}
        + a\Bigl[(1-\gamma_k)\Delta_k +\frac{L^2}{\gamma_k}R_k+\gamma_k^2\sigma^2\Bigr]
        \\
        &=
        F_k - H_k R_k + a\gamma_k^2\sigma^2 + \left(1-\gamma_k+\frac{1}{a( M_k -L)} \right)a\Delta_k,
    \end{align*}
    where $H_k \eqdef \frac{ M_k -L}{4}-\frac{aL^2}{\gamma_k}$.
    If $H_k>0$, then we can plug \Cref{lem:gradf-bound} in and get:
    \begin{align*}
        \Phi_{k+1}
        &\leq
        F_k
        -
        \frac{H_k}{M_k^2+L^2} \left( \frac{1}{3} \Eb{\norm{\nabla F(\xx_{k+1})}^2} - \Delta_k \right)
        + \left(1-\gamma_k+\frac{1}{a( M_k -L)} \right)a\Delta_k  +a\gamma_k^2\sigma^2\\
        &=
        F_k
        -
        \frac{H_k}{3 (M_k^2+L^2)} \Eb{\norm{\nabla F(\xx_{k+1})}^2}
        +
        \left( 1-\gamma_k + \frac{1}{a( M_k -L)} + \frac{H_k}{a( M_k ^2+L^2)} \right) a\Delta_k
        +
        a\gamma_k^2\sigma^2.
    \end{align*}
    Now let us choose $a=\frac{\gamma_k( M_k -L)}{8L^2}$, so that $H_k = \frac{M_k - L}{8} > 0$.
    Then,
    \[
        \frac{1}{a( M_k -L)}+\frac{H_k}{a( M_k ^2+L^2)}
        =\frac{8L^2}{\gamma_k( M_k -L)^2}+\frac{L^2}{\gamma_k( M_k ^2+L^2)}
        = \frac{L^2}{\gamma_k}\left(\frac{8}{( M_k -L)^2}+\frac{1}{ M_k ^2+L^2}\right).
    \]
    Therefore we need $\frac{L^2}{\gamma_k}\bigl(\frac{8}{( M_k -L)^2}+\frac{1}{ M_k ^2+L^2}\bigr)\leq \gamma_k$
    or, equivalently,
    $\gamma_k^2\geq L^2\bigl(\frac{8}{( M_k -L)^2}+\frac{1}{ M_k ^2+L^2}\bigr)$.
    Since $( M_k -L)^2\leq  M_k ^2+L^2$, it suffices to set $\gamma_k^2=\frac{9L^2}{( M_k -L)^2}$, i.e.,
    $\gamma_k = \frac{3L}{M_k -L}$.
    Note that this requires $ M_k >4L$ as we want to keep $\gamma_k < 1$.
    For our value of $\gamma_k$, we get $a = \frac{3}{8L}$.
    Putting everything together, we obtain
    \[
        \Phi_{k+1} \leq \Phi_k -\frac{ M_k -L}{24( M_k ^2+L^2)}\Eb{\norm{\nabla F(\xx_{k+1})}^2} + \frac{27L \sigma^2}{8( M_k -L)^2}.
    \]
    Since $\frac{ M_k -L}{M_k ^2+L^2}\geq \frac{1}{2(M_k-L)}\geq \frac{1}{2M_k}$, we can further estimate
    \[
        \Phi_{k+1} \leq \Phi_k - \frac{1}{48M_k}\Eb{\norm{\nabla F(\xx_{k+1})}^2} + \frac{27L \sigma^2}{4M_k^2}.
        \qedhere
    \]
\end{proof}

\convergenceK*

\begin{proof}
    Setting $M_k\eqdef M$ some constant, rearranging and summing \Cref{eq:lyapunov} from $0$ to $K-1$, we get
\[
    \frac{1}{48 M }\sum_{k=0}^{K-1}\Eb{\norm{\nabla F(\xx_{k+1})}^2} \leq \Phi_0 + \frac{27KL\sigma^2}{4 M ^2}
\]
Therefore 
\[
    \frac{1}{K}\sum_{k=0}^{K-1}\Eb{\norm{\nabla F(\xx_{k+1})}^2}\leq \frac{48 M \Phi_0 }{K}+ \frac{324L\sigma^2}{ M }
\]    

Recall that $ M >4L$, then for $ M  = 4L + \frac{3^{\nicefrac{3}{2}}}{2}\sqrt{\frac{KL\sigma^2}{\Phi_0}}$. We have:
\[
    \frac{1}{K}\sum_{k=0}^{K-1}\Eb{\norm{\nabla F(\xx_{k+1})}^2}\leq \frac{192L\Phi_0}{K} + 48(3^{\nicefrac{3}{2}})\sqrt{\frac{L\Phi_0\sigma^2}{K}}
\]

Therefore, after at most $\cO(\frac{L\Phi_0\sigma^2}{\epsilon^2}+\frac{L\Phi_0}{\epsilon})$ iterations, we have $\frac{1}{K}\sum_{k=0}^{K-1}\Eb{\norm{\nabla F(\xx_{k+1})}^2}\leq \epsilon$.

If $t$ is chosen uniformly at random from $\{1,\ldots, K\}$, then $\Eb{\norm{\nabla F(\xx_t)}^2}=\frac{1}{K}\sum_{k=0}^{K-1}\Eb{\norm{\nabla F(\xx_{k+1})}^2}$.
\end{proof}

We also discuss the case where $M_k$ is not a constant:
\convergenceKIncreasingMk*

\begin{proof}
    Denote $G_{k + 1}^2 \eqdef \Eb{\norm{\nabla F(\xx_{k+1})}^2}$.
    According to \Cref{lem:lyapunov}, we have
    \[
        \Phi_{k+1} \leq \Phi_k  - \frac{1}{48 M_k } G_{k + 1}^2 + \frac{27L \sigma^2}{4 M_k ^2}.
    \]
Rearranging and summing the above from $0$ to $K-1$, where $K \geq 1$ is arbitrary, we get:
\[
    \frac{1}{48} \sum_{k=0}^{K-1}\frac{1}{M_k} G_{k + 1}^2 \leq \Phi_0  + \frac{27L\sigma^2}{4}\sum_{k=0}^{K-1}\frac{1}{M_k^2}.
\]
Denoting $A_i = \sum_{k=0}^{i - 1}\frac{1}{M_k}$, we obtain
\[
    \frac{1}{ A_K} \sum_{k=0}^{K-1}\frac{1}{M_k} G_{k + 1}^2 \leq \frac{48\Phi_0  + 324L\sigma^2\sum_{k=0}^{K-1}\frac{1}{M_k^2}}{A_K},
\]
Hence, for $M_k\eqdef \max\bigl\{\sqrt{\frac{(k+1)L\sigma^2}{\phi_0}},4L\bigr\}$, we have
\[
    A_K
    = \sum_{k=1}^{K} \min\biggl\{\sqrt{\frac{\Phi_0}{kL\sigma^2}},\frac{1}{4L}\biggr\}
    \geq K\min\biggl\{\sqrt{\frac{\Phi_0}{KL\sigma^2}},\frac{1}{4L}\biggr\}
    = \min\biggl\{\sqrt{\frac{K\Phi_0}{L\sigma^2}},\frac{K}{4L}\biggr\}
\]
and
\[
    48\Phi_0+324L\sigma^2\sum_{k=0}^{K-1} \frac{1}{M_k^2}
    \leq
    48\Phi_0+324L\sigma^2\sum_{k=1}^{K} \frac{\Phi_0}{kL\sigma^2}
    \leq 372\Phi_0\sum_{k=1}^{K}\frac{1}{k}.
\]
Putting everything together, we get:
\[
    \frac{1}{ A_K} \sum_{k=0}^{K-1}\frac{1}{M_k} G_{k + 1}^2
    \leq \frac{372\Phi_0\sum_{k=1}^{K-1}\frac{1}{k}}{\min\{\sqrt{\frac{K\Phi_0}{L\sigma^2}},\frac{K}{4L}\}}
    \leq \frac{372\Phi_0 (\ln K + 1)}{\min\{\sqrt{\frac{K\Phi_0}{L\sigma^2}},\frac{K}{4L}\}}
    \leq 372 (\ln K+1) \Bigl( \sqrt{\frac{L\Phi_0\sigma^2}{K}} + \frac{L\Phi_0}{K} \Bigr).
\]
If $t(k)$ is chosen from $\{1,\ldots,k\}$ with probabilities $\Pr(t(k)=i)=\frac{1}{M_i A_k}$ , then $\Eb{\nabla F(\xx_{t(k)})} = \frac{1}{A_k}\sum_{i=0}^{k-1}\frac{G_{i+1}^2}{M_i}$.
\end{proof}

\section{Missing Proofs in Section~\ref{sec:vr-effect}}
\label{sec:proofs-vr-effect}
\begin{lemma}
    Under \cref{assumption:psi,assumption:smoothness,assumption:noise,assumption:exact-proximal-step}, for $\{\xx_k\}_{k>0}$ generated by \Cref{alg:composite-momentum}, if $\gamma_k \eqdef \frac{3\sqrt{2}L}{M_k -L}$, $a \eqdef \frac{\sqrt{2}}{8L}$, and $ M_k >(1+3\sqrt{2})L$, we have:
    \[
        \Phi_{k+1} \leq\Phi_k  - \frac{1}{48 M_k }\Eb{\norm{\nabla F(\xx_{k+1})}^2} - \frac{3\sqrt{2}}{2 M_k }\Delta_k + \frac{27\sqrt{2}L \sigma^2}{4 M_k ^2}.
    \]
\end{lemma}

\begin{proof}
    We repeat the proof of \Cref{lem:lyapunov} but now require that
    $\frac{L^2}{\gamma_k}\bigl(\frac{8}{( M_k -L)^2}+\frac{1}{ M_k ^2+L^2}\bigr)\leq \frac{\gamma_k}{2}$.
    To satisfy this inequality, it suffices to choose $\gamma_k = \frac{3\sqrt{2}L}{ M_k -L}$,
    which requires that $ M_k > (1+3\sqrt{2})L$.
    Now we have $a= \frac{\sqrt{2}}{8L}$ and:
    \[
        \Phi_{k+1} \leq \Phi_k -\frac{ M_k -L}{24( M_k ^2+L^2)}\Eb{\norm{\nabla F(\xx_{k+1})}^2} - \frac{3\sqrt{2}L}{2( M_k -L)}\Delta_k  + \frac{27\sqrt{2}L \sigma^2}{4( M_k -L)^2}.
    \]
    Let $ M_k =\tau_k L$ for some $\tau_k>(1+3\sqrt{2})$.
    Note that, for $\tau_k>(1+3\sqrt{2})$, we have $\frac{\tau_k - 1}{\tau_k^2+1}\geq\frac{1}{2\tau_k}$, $\frac{1}{2(\tau_k-1)}\geq \frac{1}{2\tau_k}$ and $\frac{1}{(\tau_k-1)^2}\leq \frac{2}{\tau_k^2}$. Therefore, for $ M_k >(1+3\sqrt{2})L$:
    \[
        \Phi_{k+1} \leq \Phi_k  - \frac{1}{48 M_k }\Eb{\norm{\nabla F(\xx_{k+1})}^2} - \frac{3\sqrt{2}}{2 M_k }\Delta_k + \frac{27\sqrt{2}L \sigma^2}{4 M_k ^2}.
        \qedhere
    \]
\end{proof}

\section{Missing Proofs in Section~\ref{sec:inexact-proximal-step}}
\label{sec:proofs-inexact}
First, we notice that, under \cref{assumption:psi,assumption:smoothness}, \cref{lem:descent-F} still holds in the inexact case. Now we give an analogous result to \Cref{lem:gradf-bound} in the inexact case:

\begin{lemma}
    \label{lem:gradf-bound-approx}
    Under \Cref{assumption:smoothness}, for $\{\xx_k\}_{k>0}$ generated by \Cref{alg:composite-momentum}, we have:
    \[
        (M_k^2+L^2) R_k
        \geq
        \frac{1}{4} \Eb{\norm{\nabla F(\xx_{k+1})}^2} -\Delta_k - \Eb{\norm{\nabla \Omega(\xx_{k+1})}^2}.
    \]
\end{lemma}

\begin{proof}
    Indeed,
    \begin{align*}
        \norm{\nabla F(\xx_{k+1})}^2 &= \norm{\nabla  f(\xx_{k+1})+\nabla\psi(\xx_{k+1})}^2\\
            &= \norm{\mm_{k}+\nabla\psi(\xx_{k+1})+(\nabla  f(\xx_k)-\mm_{k})+(\nabla  f(\xx_{k+1})-\nabla f(\xx_k))}^2\\
            &= \norm{ M_k (\xx_k-\xx_{k+1}) + \nabla \Omega(\xx_{k+1})+(\nabla  f(\xx_k)-\mm_{k})+(\nabla  f(\xx_{k+1})-\nabla f(\xx_k))}^2\\
            &\leq 4 M_k ^2\norm{\xx_{k+1}-\xx_k}^2 + 4\norm{\nabla \Omega(\xx_{k+1})}^2 + 4\norm{\mm_{k}-\nabla  f(\xx_k)}^2+4\norm{\nabla  f(\xx_{k+1})-\nabla f(\xx_k)}^2\\
            &\leq 4( M_k ^2+L^2)\norm{\xx_{k+1}-\xx_k}^2 + 4\norm{\mm_{k}-\nabla  f(\xx_k)}^2 + 4\norm{\nabla \Omega(\xx_{k+1})}^2.
    \end{align*}
    Taking expectations and rearranging, we obtain the claim.
\end{proof}

Notice that \Cref{lem:descent-Delta} still holds. We continue to use the same Lyapunov function $\Phi_k \eqdef F_k+a\Delta_k $.

\begin{lemma}
    \label{lem:lyapunov-approx}
    Under \cref{assumption:psi,assumption:smoothness,assumption:noise}, for $\{\xx_k\}_{k>0}$ generated by \Cref{alg:composite-momentum}, if $\gamma_k \eqdef \sqrt{\frac{456}{17}}\frac{L}{ M_k -L}$, $a\eqdef \sqrt{\frac{19}{408}}\frac{1}{L}$, that in each iteration condition~\eqref{eq:inexactness-criteria} holds.
    Then, for $ M_k >7L$, we have:
    \begin{equation}
        \label{eq:lyapunov-approx}
        \Phi_{k+1} \leq \Phi_k  - \frac{1}{68 M_k }\Eb{\norm{\nabla F(\xx_{k+1})}^2} + \frac{12L \sigma^2}{ M_k ^2} + \frac{2 S_k}{17 M_k}.
    \end{equation}
\end{lemma}

\begin{proof}
    Denote $G_{k + 1}^2 \eqdef \Eb{\norm{\nabla F(\xx_{k+1})}^2}$.
    Plugging \Cref{eq:inexactness-criteria} into \Cref{lem:gradf-bound-approx}, we obtain
    \[
        (M_k^2 + L^2) R_k
        \geq
        \frac{1}{4} G_{k + 1}^2 - \Delta_k - \Eb{\norm{\nabla \Omega(\xx_{k+1})}^2}
        \geq
        \frac{1}{4} G_{k + 1}^2 - \Delta_k - \frac{M_k^2}{16} R_k - S_k.
    \]
    Rearranging, we get:
    \begin{equation}
        \label{eq:gradf-bound-approx}
        \frac{17}{16} (M_k ^2+L^2) R_k
        \geq
        \frac{1}{4} G_{k + 1}^2 - \Delta_k - S_k.
    \end{equation}

    Recall in the proof of \Cref{lem:lyapunov}, we have: 
    \[
        \Phi_{k+1} 
        \leq
        F_k-H_k R_k+ \left(1-\gamma_k+\frac{1}{a( M_k -L)} \right)a\Delta_k + a\gamma_k^2\sigma^2,
    \]
    where $H_k \eqdef \frac{ M_k -L}{4}-\frac{aL^2}{\gamma_k}$.
    If $H_k > 0$, then we can plug \Cref{eq:gradf-bound-approx} in and get:
    \begin{align*}
        \Phi_{k+1}
        &\leq
        F_k
        -
        \frac{16 H_k}{17( M_k ^2+L^2)} \left( \frac{1}{4} G_{k + 1}^2 - \Delta_k - S_k \right)
        +
        \left(1-\gamma_k+\frac{1}{a( M_k -L)} \right)a\Delta_k  +a\gamma_k^2\sigma^2 \\
        &= F_k - \frac{4 H_k}{17 (M_k^2 + L^2)} G_{k + 1}^2
        + \left(1-\gamma_k+\frac{1}{a( M_k -L)}+\frac{16 H_k}{17 a( M_k ^2+L^2)} \right)a\Delta_k
        + a\gamma_k^2\sigma^2
        + \frac{16 H_k S_k}{17( M_k ^2+L^2)}.
    \end{align*}
    Now let $a=\frac{\gamma_k( M_k -L)}{8L^2}$, so that $H_k = \frac{M_k - L}{8}$.
    Then,
    \[
        \frac{1}{a( M_k -L)}+\frac{16 H_k}{17 a( M_k ^2+L^2)}
        =\frac{8L^2}{\gamma_k( M_k -L)^2}+\frac{16L^2}{17\gamma_k( M_k ^2+L^2)}
        = \frac{8L^2}{\gamma_k}\left(\frac{1}{( M_k -L)^2}+\frac{2}{17( M_k ^2+L^2)}\right).
    \]
    Therefore we need $\frac{8L^2}{\gamma_k}\bigl(\frac{1}{( M_k -L)^2}+\frac{2}{17( M_k ^2+L^2)}\bigr)\leq \gamma_k$
    or, equivalently,
    $\gamma_k^2\geq 8L^2\bigl(\frac{1}{( M_k -L)^2}+\frac{2}{17( M _k^2+L^2)}\bigr)$.
    Since $( M_k -L)^2\leq  M_k ^2+L^2$, it suffices to set $\gamma_k^2=\frac{152}{17}\frac{L^2}{( M_k -L)^2}$, i.e.,
    $\gamma_k = \sqrt{\frac{152}{17}}\frac{L}{ M_k -L}$.
    Note that this requires $ M_k >\frac{(\sqrt{152}+\sqrt{17})L}{\sqrt{17}}$, which is ensured whenever $ M_k> 4L$.
    For our choice of $\gamma_k$, we get
    $a = \frac{1}{8L} \sqrt{\frac{152}{17}}= \sqrt{\frac{19}{156}}\frac{1}{L}$.
    Putting everything together, we get
    \[
        \Phi_{k+1} \leq \Phi_k -\frac{ M_k -L}{34( M_k ^2+L^2)} G_{k + 1}^2 + \frac{4L}{( M_k -L)^2}\sigma^2  + \frac{2( M_k -L)S_k}{17( M_k ^2+L^2)}.
    \]
    Let $ M_k =\tau_k L$ for some $\tau_k>7$.
    Note that $\frac{\tau_k - 1}{\tau_k^2+1}\geq\frac{1}{2\tau_k}$, $\frac{1}{(\tau_k-1)^2}\leq \frac{2}{\tau_k^2}$ and $\frac{\tau_k-1}{\tau_k^2+1}\leq \frac{1}{\tau_k}$.
    Therefore, for $M_k >7L$, we get
    \[
        \Phi_{k+1} \leq \Phi_k  - \frac{1}{68 M_k } G_{k + 1}^2 + \frac{8L}{ M_k ^2}\sigma^2 + \frac{2 S_k}{17 M_k }.
        \qedhere
    \]
\end{proof}

Setting $M_k=M$ some constant, and rearranging and summing \Cref{eq:lyapunov-approx} over $k=0,\ldots,K-1$, we get:
\[
    \frac{1}{68 M } \sum_{k=0}^{K-1}\Eb{\norm{\nabla F(\xx_{k+1})}^2} \leq \Phi_0  + \frac{8KL}{ M ^2}\sigma^2 + \frac{2}{17 M }\sum_{k=0}^{K-1}S_k.
\]
Therefore,
\[
    \frac{1}{K}\sum_{k=0}^{K-1}\Eb{\norm{\nabla F(\xx_{k+1})}^2} \leq \frac{68 M \Phi_0}{K} + \frac{544L\sigma^2}{ M }+ \frac{8}{K}\sum_{k=0}^{K-1}S_k.
\]
Then for $ M =7L + \sqrt{\frac{8KL\sigma^2}{\Phi_0}}$, we have:
\[
    \frac{1}{K}\sum_{k=0}^{K-1}\Eb{\norm{\nabla F(\xx_{k+1})}^2} \leq \cO\biggl(\sqrt{\frac{L\Phi_0\sigma^2}{K}} + \frac{L\sigma^2}{K} \biggr) + \frac{8}{K}\sum_{k=0}^{K-1}S_k.
\]

\section{Sampling from a Stream of Data}
\label{sec:sampling}

Following \Cref{thm:convergence-exact}, we briefly mentioned that one can efficiently sample the desired output point. In this section, we explain how to perform such sampling at no extra computation and memory cost. This might have been discussed in the literature, but we still provide a detailed explanation for completeness and the reader's convenience.

\begin{proposition}
    \label{prop:sampling}
    Given a stream of points $\{\xx_k\}_{k=1}^{\infty}$ in $\R^d$ and positive scalars $\{h_k\}_{k=1}^{\infty}$, we can maintain, at each step~$k \geq 1$, the random variable $\xx_{t(k)}$, where $t(k)$ is a random index from $\{1, \ldots, k\}$ chosen with probabilities $\Pr(t(k) = i) = \frac{h_i}{H_k}$, $i = 1, \ldots, k$, where $H_k\eqdef \sum_{i=1}^k h_i$. This requires only $\cO(d)$ memory and computation.
\end{proposition}

\begin{proof}
    We maintain the variables $\hat \xx_k \in \R^d$ and $H_k \in \R$ which are both initialized to $0$ at step $k = 0$. Then, at each step~$k \geq 1$, we update $H_k \leftarrow H_{k - 1} + h_k$ and also, with probability $\frac{h_k}{H_k}$, we update $\hat \xx_k \leftarrow \xx_k$ (or, with probability $1 - \frac{h_k}{H_k}$, keep the old $\hat{\xx}_k = \hat{\xx}_{k - 1}$). The memory and computation costs are $\cO(d)$.
    Note that, for any $1 \leq i \leq k$, the event $\hat{\xx}_k = \xx_i$ happens iff $\hat{\xx}$ was updated at step $i$ and then not updated at each step $j = i + 1, \ldots, k$.
    Hence, for any $1 \leq i \leq k$, we have
    \[
        \Pr(\hat{\xx}_k = \xx_i)
        =
        \frac{h_i}{H_i} \cdot \prod_{j=i+1}^{k}\left(1-\frac{h_j}{H_j}\right) = \frac{h_i}{H_i}\cdot \prod_{j=i+1}^{k}\frac{H_{j-1}}{H_j} =\frac{h_i}{H_k}.
        \qedhere
    \]
\end{proof}

\section{Convergence in the Non-Differentiable Case}
\label{sec:non-differentiable}

In this section we briefly discuss the convergence of \Cref{alg:composite-momentum} in the non-differentiable case. In this case we assume that $\psi$ is convex. We assume that each subproblem is solved exactly, i.e. $\xx_{k+1}=\argmin_{\xx}\Omega_k(\xx)$. In particular, this implies that there exists subgradient $\gg_{k+1}^{\psi}\in \partial \psi(\xx_{k+1})$ such that:
\[
    \mm_k + \gg_{k+1}^{\psi} +M(\xx_{k+1}-\xx_k)=\0
\]
Therefore, for convenience, we define $\nabla \psi(\xx_{k+1})\eqdef \gg_{k+1}^{\psi}$ as the specific subgradient of $\psi$ at $\xx_{k+1}$ that we use. Similarly, we define $\nabla F(\xx_{k+1})\eqdef \nabla f(\xx_{k+1})+ \gg_{k+1}^{\psi}=\nabla f(\xx_{k+1}) + \nabla \psi(\xx_{k+1})$ as the specific subgradient of $F$ at $\xx_{k+1}$ that we use. This way, all our previous proofs still hold, and we have the following:

\begin{theorem}
    \label{thm:non-differentiable}
    Under \cref{assumption:psi,assumption:smoothness,assumption:noise}, for $\{\xx_k\}_{k>0}$ generated by \Cref{alg:composite-momentum}, if $M_k=M$ some constant, $\gamma \eqdef \frac{3L}{ M -L}$, there exists $M=4L+\frac{3^{\nicefrac{3}{2}}}{2}\sqrt{\frac{KL\sigma^2}{\Phi_0}}$ such that
    $K = \cO\bigl( \frac{L\Phi_0\sigma^2}{\epsilon^2}+\frac{L\Phi_0}{\epsilon} \bigr)$
    iterations of \Cref{alg:composite-momentum} is sufficient to achieve $\Eb{\norm{\nabla F(\xx_{t})}^2} \leq \epsilon$, where $t$ is chosen from $\{1,\ldots ,K\}$ uniformly at random.
\end{theorem}
We remark that this can be reformulated in terms of the distance between $\partial F$ and $\0$: we have ${\rm dist}^2(\partial F(\xx_{t}),0)\leq \epsilon$.

\section{Convergence Criterias}
\label{sec:convergence-criteria}

In this section, we discuss the differences and connections between the convergence criterias used in the literature.

\paragraph*{Proximal Gradient Mapping:} The most popular convergenec criteria used in most of the earlier works on non-convex composite optimization is proximal gradient mapping~\cite{ghadimi2016mini, ghadimi2013accelerated, wang2019spiderboost,hendrikx2020statistically, tran2022hybrid,xu2022momentumbased}. Proximal gradient mapping is defined in a very general context, where we consider the constrained composite optimization problem:
\[
    \min_{\xx\in X} \bigl[ F(\xx) \eqdef f(\xx) + \psi(\xx) \bigr],
\]
with the $1$-strongly convex mirror map $r:X\to \R$. For any vector $\gg$ and $\xx$, and scalar $ M >0$, define:
\[
    \xx^+ \eqdef \argmin_{\xx'\in X}\{\inp{\gg}{\xx'-\xx}+ \psi(\xx)+ \beta_r(\xx',\xx)\}
\]
where $\beta_r(\xx',\xx) = r(\xx')-r(\xx)-\inp{\nabla r(\xx)}{\xx'-\xx}$ is the Bregman divergence of $r$. Then the proximal gradient mapping is defined as:
\[
    P_X(\xx,\gg, M ) \eqdef  M (\xx-\xx^+),
\]
and in the literature, the convergence is studied in terms of $\norm{P_X(\xx,\nabla f(\xx), M )}^2$. We point out that it is very easy to prove analogous versions of \cref{lem:descent-F,lem:gradf-bound} in terms of $\norm{P_X(\xx,\nabla f(\xx), M )}^2$ (which implies that our results extend to the proximal gradient mapping case), but we omit the details here. Instead, we argue that the size of $\nabla F(\xx)$ is a more natural convergence criterion. Consider the simplest situation where $\psi$ is convex and differentiable, $X=\R^d$, and the mirror map $r(\xx)=\frac{1}{2}\norm{\xx}^2$ (i.e. the Euclidean geometry). Then we have the usual stationarity condition for $\xx^+$:
\[
    \nabla f(\xx) + \nabla \psi(\xx^+)+ M (\xx^+-\xx) = 0.
\]
Therefore, we have:
\[
    P_X(\xx,\nabla f(\xx), M ) =  \nabla f(\xx) + \nabla \psi(\xx^+).
\]
In other words, the proximal gradient mapping gives the gradient of $f$ at the current point plus the gradient of $\psi$ at the next point. In contrast, in our analysis, we directly consider the gradient of $F$ at the next point, and we believe that this is more natural and intuitive.

\paragraph*{Moreau Envelope:} A closely related convergence criterion that was proposed to address the non-convergence problem of proximal gradient mapping of the earlier works is the Moreau envelope~\citep{davis2019stochastic}. In Euclidean geometry, for some parameter $\lambda$, the Moreau envelope is defined as the following:
\[
    F_{\lambda} (\xx) \eqdef \min_{\yy} \Bigl\{ F(\yy) + \frac{1}{2\lambda}\norm{\yy-\xx}^2 \Bigr\},
\]
and write $\hat \xx = \argmin_{\yy} \{ F(\yy) + \frac{1}{2\lambda}\norm{\yy-\xx}^2 \}$. The convergence criteria is then defined as:
\[
    \nabla F_{\lambda}(\xx) \eqdef \lambda^{-1} (\xx-\hat \xx).
\]
Equivalently, $\nabla F_{\lambda}(\xx) =\nabla F(\hat \xx)$ in the differentiable case. In other words, the convergence criteria with Moreau envelope uses a surrogate point $\hat\xx$ instead of the actual iterates of the algorithm. Note that, the convergence criteria using Moreau envelop and proximal gradient mapping is with a constant factor of each other for a $\rho$-weakly convex function~\citep{davis2019stochastic}:

\[
    \frac{1}{4}\norm{\nabla F_{\nicefrac{1}{2\rho}}(\xx)} \leq \norm{P(\xx,\nabla f(\xx), \rho )} \leq \frac{3}{2}\left(1+\frac{1}{\sqrt{2}}\right)\norm{\nabla F_{\nicefrac{1}{2\rho}}(\xx)}.
\]